\documentclass[a4paper,12pt]{article}

\usepackage{graphicx} 
\usepackage{geometry}
\usepackage{amsfonts,amsmath,amssymb,amsthm,enumitem}
\usepackage{thmtools}
\usepackage{mathcomp,mathrsfs}
\usepackage{bm}
\usepackage{tikz-cd}
\usepackage[linktoc=page]{hyperref}
\usepackage[capitalise]{cleveref}
\usepackage{comment}

\theoremstyle{definition}

\newtheorem{definition}{Definition}[section]
\newtheorem*{definition*}{Definition}
\newtheorem{notation}[definition]{Notation}
\newtheorem*{notation*}{Notation}

\newtheorem*{assumption*}{Assumption}

\theoremstyle{plain}

\newtheorem{theorem}[definition]{Theorem}
\newtheorem{theoremAlph}{Theorem}

\newtheorem*{theorem*}{Theorem}

\newtheorem*{conjecture*}{Conjecture}
\newtheorem{proposition}[definition]{Proposition}
\newtheorem*{proposition*}{Proposition}
\newtheorem{lemma}[definition]{Lemma}
\newtheorem*{lemma*}{Lemma}

\newtheorem*{corollary*}{Corollary}

\theoremstyle{remark}

\newtheorem{example}[definition]{Example}
\newtheorem{remark}[definition]{Remark}

\DeclareMathOperator{\Coker}{Coker}

\DeclareMathOperator{\fac}{fac}
\DeclareMathOperator{\Ker}{Ker}
\DeclareMathOperator{\GL}{GL}
\DeclareMathOperator{\Hom}{Hom}

\DeclareMathOperator{\Ima}{Im}

\DeclareMathOperator{\rk}{rk}

\DeclareMathOperator{\Stab}{Stab}
\DeclareMathOperator{\sub}{sub}

\newcommand{\st}{\mathrm{s}}

\newcommand\GG{\mathbb{G}}

\newcommand\gl{\mathfrak{gl}}

\newcommand\bbC{\mathbb{C}}

\newcommand\bbZ{\mathbb{Z}}
\newcommand\bfalpha{\bm{\alpha}}

\newcommand\bfd{\mathbf{d}}
\newcommand\bff{\mathbf{f}}
\newcommand\bfm{\mathbf{m}}

\newcommand\bfr{\mathbf{r}}

\newcommand\rmT{\mathrm{T}}

\newcommand\git{/\!/}

\usepackage[disable]{todonotes}

\usepackage{xcolor}

\usepackage[style=alphabetic]{biblatex}
\AtBeginBibliography{\footnotesize}

\title{Nilpotent quiver varieties with multiplicities and symmetrisable crystals}

\author{Victoria Hoskins\thanks{Universit\"{a}t Duisburg-Essen Germany, \texttt{victoria.hoskins@uni-due.de}} \and Jessica Jackson\thanks{St John's College, University of Cambridge, \texttt{jjj26@cam.ac.uk}} \and Tanguy Vernet\thanks{Universit\"{a}t Duisburg-Essen Germany, \texttt{tanguy.vernet@uni-due.de}}}
\date{}

\makeatletter
\renewcommand\tableofcontents{\@starttoc{toc}}
\makeatother

\begin{document}

\setcounter{tocdepth}{1}

\maketitle

\abstract{We construct analogues of Nakajima's nilpotent quiver varieties for quivers with multiplicities, by  developing new stability conditions for framed nilpotent quiver representations with multiplicities and exploiting their interplay with Hecke correspondences. Using these new quiver varieties, we extend Saito's geometric construction of Kashiwara crystals of irreducible highest-weight representations of Kac--Moody algebras from the symmetric to the symmetrisable case.}


\tableofcontents

\section{Introduction}

Quiver varieties, introduced by Nakajima in \cite{Nak94}, are a cornerstone of the geometric representation theory of Kac--Moody algebras. Given a \emph{symmetric} Kac--Moody algebra $\mathfrak{g}$, one obtains a quiver by adding an orientation to its Dynkin diagram. Nakajima showed that highest-weight irreducible representations of $\mathfrak{g}$ can be reconstructed from top-degree homology groups of so-called nilpotent quiver varieties associated to that quiver \cite{Nak98}. The Kashiwara crystals \cite{Kas91} of these irreducible representations were subsequently constructed by Saito from irreducible components of the same nilpotent quiver varieties \cite{Sai02}. These results were further generalised to symmetric Borcherds and Borcherds-Bozec algebras in \cite{KKS12,Boz16}.

These results are now understood as part of several interrelated geometric constructions of Kac--Moody algebras. Nilpotent quiver varieties are linked to Ringel-Hall algebras of quivers via a microlocal construction \cite{Rin90a,Lus91,KS97,Hen24a,FL25}, and the action of $\mathfrak{g}$ on their homology extends to an action of a BPS Lie algebra built from the preprojective cohomological Hall algebra of the quiver \cite{SV13b,YZ18a,Dav25a,DHSM23}. Poincaré polynomials of (nilpotent) quiver varieties are also related to Kac polynomials of quivers and can be expressed in terms of Cartan data of Kac--Moody algebras \cite{CBVdB04,HLRV13,Hau10,BS19,BSV20,Dav23c,DHSM23}.

In this paper, we are concerned with geometric realisations of representations of a \emph{symmetrisable} Kac--Moody algebra $\mathfrak{g}$.  Some of the above constructions have been generalised to the symmetrisable case over the past decade. In order to account for Cartan data with non-trivial symmetriser, one should consider quivers with additional structure. Geiss--Leclerc--Schr\"{o}er obtained a realisation of the enveloping algebra of $\mathfrak{g}$ by generalising the Ringel--Hall approach to quivers with multiplicities \cite{GLS16}. Moreover, the construction of the Kashiwara crystal attached to the quantised universal enveloping algebra of $\mathfrak{g}$ \cite{KS97} was extended to the symmetrisable case in several works, using graphs with automorphisms \cite{Sav05}, species \cite{NT18} and quivers with multiplicities \cite{GLS18a}. Geometric realisations of the Yangian and of the (shifted) quantum loop group associated to $\mathfrak{g}$ were also provided by Yang--Zhao, Varagnolo--Vasserot and Cao--Okounkov--Zhou--Zhou using quivers with potential \cite{YZ22,VV23b,COZZ26a}. These quivers with potential appear in \cite{HL16,Neg25b} in a related context.

Following Geiss--Leclerc--Schr\"{o}er's approach \cite{GLS18a}, we give a geometric construction of the crystals of irreducible highest-weight representations of any \emph{symmetrisable} Kac--Moody algebra $\mathfrak{g}$ (see \cref{MainThm/crystalConstruction} below). This builds on the results of Geiss--Leclerc--Schr\"{o}er \cite{GLS18a} in a similar way to how Saito's construction \cite{Sai02} expands on his earlier work with Kashiwara \cite{KS97}.

To carry out our construction, we introduce analogues of Nakajima's nilpotent quiver varieties \cite{Nak94} for quivers with multiplicities (see \cref{MainThm/quotientConstruction} below). These quiver varieties are constructed in the spirit of Geometric Invariant Theory (GIT), as quotients of stable loci in the representation spaces of a framed quiver with multiplicities associated to $\mathfrak{g}$ and a fixed highest weight by a \emph{non-reductive} group. We identify a stability condition on quiver representations with multiplicities, which allows us to generalise Saito's construction \cite{Sai02} to the symmetrisable case. In our previous work on moduli of quiver representations with multiplicities \cite{HJV25}, we introduced new stability conditions and constructed moduli spaces of quiver representations with multiplicities in a general setup, using tools from non-reductive GIT \cite{HHJ}. However, these stability conditions coming from GIT do not fit our needs for the construction of crystals, so we instead define a new notion of stability and recursively construct nilpotent quiver varieties using slices. Our construction of slices relies on so-called \emph{Hecke correspondences} (see below), which parametrise \emph{Hecke modifications} of quiver representations at a vertex and induce \emph{Hecke (or crystal) operators} on the set of irreducible components of nilpotent quiver varieties.

\paragraph*{Main results}

We fix a quiver with multiplicities $(Q,\bfm)$, where $\bfm$ is a collection of positive integers labelled by the vertices of $Q$. We assume throughout that $Q$ has no loop arrows. Let $\bfr$ be a rank vector and $\bff$ be a framing vector. We consider the framed quiver $Q_\bff$ and the analogue $\Lambda_\bfr:=\Lambda_{Q_\bff,\widehat{\bfm},\widehat{\bfr}}$ of Lusztig's nilpotent cone \cite{Lus91} for quivers with multiplicities, which was introduced by Geiss--Leclerc--Schröer \cite{GLS18a}. This is a (very reducible) lagrangian subscheme of the cotangent bundle of the representation space of $(Q_\bff,\widehat{\bfm})$. We refer to \cref{Sect/prerequisitesQuiverModuli} for precise definitions.

Given the above data, we construct a nilpotent quiver variety as the quotient of a certain stable locus $\Lambda_\bfr^\st\subset\Lambda_\bfr$ under the action of a group $\GL_{\bfm,\bfr}$, which is a product of general linear groups with coefficients in rings of truncated power series. This stability condition is similar to Nakajima's original stability condition for classical quiver varieties \cite{Nak94}. The following result is proved in \cref{Sect/constructionNilpQuiverVar}:

\begin{theoremAlph}[\cref{Prop/trivialStabilisers}, \cref{Thm/quotientConstruction}]\label{MainThm/quotientConstruction}
The quotient stack $\left[\Lambda_\bfr^\st/\GL_{\bfm,\bfr}\right]$ is an algebraic space. Moreover, for each irreducible component $Z\subset\Lambda_\bfr^\st$, there is an \emph{explicit} dense open $\GL_{\bfm,\bfr}$-invariant subset $Z_{\mathrm{filt}}^\st\subset Z$ (see \cref{def/stable}) which admits a geometric quotient, such that the quotient map is a Zariski-locally trivial principal $\GL_{\bfm,\bfr}$-bundle.
\end{theoremAlph}

In fact, we prove that the ambient moduli stack of stable representations of the doubled framed quiver $\overline{Q_\bff}$ is an algebraic space in \cref{Prop/trivialStabilisers}. As mentioned above, the proof of the second statement in \cref{MainThm/quotientConstruction} is by induction on $\bfr$, and our inductive proof relies heavily on the fact that we are working inside the analogue of the nilpotent cone, so representations in $Z_{\mathrm{filt}}$ are all \emph{E-filtered} (i.e.\ built as iterative extensions of representations concentrated at a single vertex). As observed in \cite{Lus91,GLS18a}, the irreducible components of $\Lambda_\bfr$ can be obtained from those of $\Lambda_{\bfr'}$, for various smaller rank vectors $\bfr'<\bfr$, via so-called Hecke correspondences:
\[
\Lambda_{\bfr'}\overset{q}{\longleftarrow}\Lambda_{\bfr',\bfr}\overset{p}{\longrightarrow}\Lambda_\bfr
,
\]
where $\Lambda_{\bfr',\bfr}$ parametrises nested representations of $(Q_\bff,\widehat{\bfm})$. Roughly speaking, given an irreducible component $Z\subset\Lambda_\bfr$, we can find $\bfr'<\bfr$ such that $p$ restricts to a principal bundle over an open subset $Z_{\mathrm{filt}}\subset Z$ (see \cref{Rmk/HeckeCorrespondences} for more precise statements). We can then construct the non-reductive quotient $Z_{\mathrm{filt}}^\st/\GL_{\bfm,\bfr}$ in a relative fashion, by analysing the map $q:p^{-1}(Z_{\mathrm{filt}}^\st)\rightarrow\Lambda_{\bfr',\mathrm{filt}}^\st$ and exploiting the existence of the quotient $(Z')_{\mathrm{filt}}^\st/\GL_{\bfm,\bfr'}$ for a certain irreducible component $Z'\subset\Lambda_{\bfr'}$ corresponding to $Z$. This boils down to constructing certain partial flag varieties (over rings of truncated power series) as relative quotients, by providing explicit slices (see \cref{Sect/toyModel}). Roughly speaking, these flag varieties parametrise the vector spaces underlying nested quiver representations in $\Lambda_{\bfr',\bfr}$.

This iteration of relative quotients produces new stability conditions compared to our previous work \cite{HJV25}, but also more intricate quotients. It would be interesting to describe the geometry of these quiver varieties more precisely (embeddings in projective spaces over a base, cell decompositions, etc.). We hope to return to this in future works.

We now explain our results concerning geometric realisations of crystals. One can attach to $(Q,\bfm)$ a symmetrisable Cartan matrix $C$ and the corresponding Kac--Moody algebra $\mathfrak{g}$. Let $\lambda$ be the integral dominant weight whose coordinates in the basis of fundamental weights are given by $\bff$. Then there is an irreducible module $L(\lambda)$ of highest weight $\lambda$ over the quantised universal enveloping algebra $U_q(\mathfrak{g})$. Let $B(C;\lambda)$ be the Kashiwara crystal attached to $L(\lambda)$ \cite{Kas91}; this can be viewed as a combinatorial skeleton of $L(\lambda)$. Its elements correspond to the vectors of the crystal basis of $L(\lambda)$ \cite{Kas91,Lus91,GL93}.

Following \cite{Sai02,GLS18a}, in \cref{Sect/prerequisitesCrystals}, we endow the following set of irreducible components in the stable locus with a Kashiwara crystal structure:
\[
B_\bff^\st:=\bigsqcup_{\bfr}\mathrm{Irr}(\Lambda_\bfr^\st).
\]
The crystal structure is obtained using the above Hecke correspondences $\Lambda_{\bfr',\bfr}$, which were introduced by Geiss--Leclerc--Schr\"{o}er in \cite{GLS18a}. Our second main result is then: 

\begin{theoremAlph}[\cref{Thm/geometricHWcrystal}]\label{MainThm/crystalConstruction}
There is an isomorphism of crystals $B_\bff^\st\simeq B(C;\lambda)$.
\end{theoremAlph}

Our proof follows the strategy of \cite{Sai02}. By forgetting the framing arrows of $Q_\bff$, we obtain a morphism of crystals relating $B_\bff^\st$ and the geometric realisation of the crystal of $U_q^-(\mathfrak{g})$ built by Geiss--Leclerc--Schr\"{o}er in \cite{GLS18a}. We then show that, under this morphism, our stability condition selects precisely the irreducible components of $\Lambda_\bfr$ corresponding to the crystal basis vectors of $U_q^-(\mathfrak{g})$ which are mapped to $B(C;\lambda)$ via the homomorphism $U_q^-(\mathfrak{g})\twoheadrightarrow L(\lambda)$. The latter statement is proved in \cref{Sect/GeometricHWCrystals}.

Since the irreducible components of nilpotent quiver varieties with multiplicities match the crystals of the representations $L(\lambda)$, it is natural to ask whether one can construct the representations themselves from the Borel--Moore homology of quiver varieties in the symmetrisable case, as Nakajima did in the symmetric case \cite{Nak98}. Some difficulties arise in the case of quivers with multiplicities due to the lack of a convenient proper pushforward map in homology. We plan to return to this question in future works.

\paragraph*{Relation to other works}

Varagnolo--Vasserot and Cao--Okounkov--Zhou--Zhou construct certain representations of non-symmetric (shifted) quantum loop groups and Yangians from critical quiver varieties attached to a tripled quiver with potential and a stability condition from reductive GIT \cite{VV23b,COZZ26a}. Cyclic derivatives of this potential yield the defining equations of our nilpotent quiver varieties, as noted in \cite{GLS17a}. However, there seems to be no straightforward relation between our quiver varieties (nor between the topological invariants we study). Indeed, these authors use critical K-theory (or cohomology), whereas our approach is related to usual (top-degree) Borel--Moore homology and those do not relate by dimensional reduction in the case of quivers with multiplicities (see \cite{Dav17,Pad22a,VV22} for comparison in the case of quivers without multiplicities). Moreover, working with locally free representations amounts to prescribing certain Jordan types for loop arrows in the triple quiver from \cite{VV23b,COZZ26a}, which seems at odds with usual GIT techniques.

Let us also mention that there is an analogue of nilpotent quiver varieties in the setting of Higgs bundles on a smooth projective curve: the global nilpotent cone \cite{Lau88}. Irreducible components of the global nilpotent cone and Hecke operators also play an important r\^{o}le in that setting: for instance in the study of CoHAs of Higgs sheaves \cite{Min20,SS20} and in mirror symmetry \cite{HH22}.
Spaces of Hecke modifications can be viewed as non-reductive quotients in that setup as well, as in the forthcoming thesis \cite{Mar26}. In this work, the author studies the nilpotent cone in moduli spaces of Higgs bundles, by considering Hecke modifications of rank two vector bundles on a smooth projective curve, along a suitably chosen non-reduced divisor.

\paragraph*{Notation and conventions.}
Throughout the paper, we work over an algebraically closed base field $k$ of arbitrary characteristic. Given a constructible subset $X$ of a scheme, we call $\mathrm{Irr}(X)$ the set of irreducible components of $X$. Given an action of an algebraic group $G$ on a scheme $X$, we say that $X$ admits a Zariski-locally trivial $G$-quotient if it admits a categorical quotient (in the category of schemes) such that the quotient map is a Zariski-locally trivial principal $G$-bundle. General definitions in Sections \ref{Sect/QuivMultDefinitions}-\ref{Sect/framedQuivRep} are given for arbitrary quivers $Q$. From \cref{Sect/CrystalComponents} on, we always assume that $Q$ has no loop arrows.

\paragraph*{Acknowledgements.}
We would like to thank Olivier Schiffmann, Michela Varagnolo and \'{E}ric Vasserot for helpful discussions on the topics of the paper. We also thank Jiarui Fei for pointing out a few inaccuracies in a previous version of this paper. T.V. was supported by the European Union’s Horizon 2020 research and innovation programme under the Marie Skłodowska-Curie Grant Agreement No. 101034413. He thanks Tam\'{a}s Hausel for his support. T.V. was also supported by the Deutsche Forschungsgemeinschaft (DFG, German Research Foundation) - Project No. 585538853.

\paragraph*{Use of generative AI.}
No generative AI was used to write the first version of this article or produce mathematical arguments therein. The Large Language Models ChatGPT 6 Astra and Claude Opus 5.5 were used to find and fix minor inaccuracies in the first version.

\section{Prerequisites on quiver moduli}\label{Sect/prerequisitesQuiverModuli}

In this section, we gather background definitions and results concerning representations of quivers with multiplicities, framings and E-filtered representations of the associated preprojective algebras, following \cite{GLS18a,HJV25}.

\subsection{Quiver with multiplicities and moment maps}\label{Sect/QuivMultDefinitions}

We recall basic definitions concerning quivers with multiplicities and the associated moment maps. We refer to \cite{GLS17a,Ver24b,HJV25} for a more detailed discussion.

Let $Q$ be a quiver with vertex set $Q_0$ and arrow set $Q_1$. We will call $s,t:Q_1\rightarrow Q_0$ the source and target maps. We fix a collection of positive integers $\bfm=(m_i)_{i\in Q_0}$, called multiplicities, and refer to $(Q,\bfm)$ as a quiver with multiplicities.

\begin{notation}\label{notation gcd of multiplicities}
For $m \in \mathbb{Z}_{>0}$, let $k_m:=k[\epsilon]/(\epsilon^m)$. For multiplicities $\bfm=(m_i)_{i\in Q_0}$, we introduce the following notation (consistent with \cite{GLS17a}) for $i,j\in Q_0$:
\begin{itemize}
    \item $m_{ij}:=\gcd(m_i,m_j)$,
    \item $f_{ji}:=\frac{m_i}{m_{ij}}$.
\end{itemize}
We view  $k_{m_i}$ as a $k_{m_{ij}}$-algebra via the $k$-algebra homomorphism $k_{m_{ij}} \rightarrow k_{m_i}$ given by $\epsilon\mapsto\epsilon^{f_{ji}}$.
For an arrow $a \colon i\rightarrow j$ in $Q$, we define the following $k_{m_i}-k_{m_j}$-bimodule:
\[
k_a:=k_{m_i}\otimes_{k_{m_{ij}}}k_{m_j}.
\]
\end{notation}

\begin{definition}\label{def/quiver rep with mult}
A representation of $(Q,\bfm)$ is a tuple $M = (M_i, i \in Q_0;M_a, a \in Q_1)$ where $M_i$ are $k_{m_i}$-modules and for $a \colon i \rightarrow j$, the map $M_a \colon M_{i}\rightarrow M_{j}$ is $k_{m_{ij}}$-linear. We say $M$ is locally free and finitely generated if the $k_{m_i}$-module $M_i$ is free and finitely generated for all $i \in Q_0$. In this case, the rank vector of $M$ is $\rk(M):=(\rk M_i)_{i\in Q_0}$.

A morphism $f \colon M\rightarrow M'$ between representations of $(Q,\bfm)$ is a collection of $k_{m_i}$-linear maps $f_i \colon M_i\rightarrow M'_i$ which commute with the arrow morphisms.
\end{definition}

In this paper, we will mostly consider locally free representations of finite rank. For $i\in Q_0$, we let $\bfalpha_i\in\bbZ^{Q_0}$ be the unit coordinate rank vector concentrated at vertex $i$.

\begin{definition}\label{Def action on rep space}
For a rank vector $\bfr\in\bbZ_{\geq0}^{Q_0}$, the representation space of $(Q,\bfm)$ with rank $\bfr$ is the following affine space over $k$:
\[
R(Q,\bfm;\bfr):=\prod_{a\in Q_1}\Hom_{k_{m_{ij}}}\left(k_{m_i}^{\oplus r_i},k_{m_j}^{\oplus r_j}\right).
\]
This space is endowed with a conjugation action of the following algebraic group:
\[
\GL_{\bfm,\bfr}:=\prod_{i\in Q_0}\GL_{m_i,r_i},
\]
where $\GL_{m,r}:= \GL_r(k_{m})$ is the automorphism group of $k_m^{\oplus r}$. 
\end{definition}

This action induces a Hamiltonian $\GL_{\bfm,\bfr}$-action on the cotangent bundle with moment map $\mu_{(Q,\bfm),\bfr} \colon \rmT^*R(Q,\bfm;\bfr)\simeq R(Q,\bfm;\bfr)\oplus R(Q,\bfm;\bfr)^\vee\rightarrow\gl_{\bfm,\bfr}^{\vee}$, which is characterised as follows:
\[
\forall (x,y,\xi)\in \mathrm{T}^*R(Q,\bfm;\bfr)\times\gl_{\bfm,\bfr},\ \langle\xi\cdot x,y\rangle=\langle\xi,\mu_{(Q,\bfm),\bfr}(x,y)\rangle
.
\]
Let $\overline{Q}$ be the doubled quiver associated to $Q$. In order to explicitly describe the moment map, we identify $\gl_{\bfm,\bfr}^{\vee}\cong \gl_{\bfm,\bfr}$ and $\mathrm{T}^*R(Q,\bfm;\bfr)\cong R(\overline{Q},\bfm;\bfr)$ using the trace pairing (see \cite[\S 8]{GLS17a}\cite[\S 1.2]{Ver24b}\cite[\S 2.6]{HJV25}).

\begin{proposition}
 Under the above identifications, the moment map is described as follows. For $(x,y)\in R(\overline{Q},\bfm;\bfr)$:
\[
\mu_{(Q,\bfm),\bfr}(x,y)=
\left(
\sum_{\substack{a\in Q_1 \\ a \colon j\rightarrow i}}
\sum_{f=0}^{f_{ji}-1}
\epsilon^{f}x_ay_a\epsilon^{f_{ji}-1-f}
-
\sum_{\substack{a\in Q_1 \\ a \colon i\rightarrow j}}
\sum_{f=0}^{f_{ji}-1}
\epsilon^{f}y_ax_a\epsilon^{f_{ji}-1-f}
\right)_{i\in Q_0}
.
\]
\end{proposition}

The quotient stack $\left[\mu_{Q,\bfm;\bfr}^{-1}(0)/\GL_{\bfm,\bfr}\right]$ is the moduli stack of modules of rank $\bfr$ over a certain algebra $\Pi_{Q,\bfm}$, called the preprojective algebra of $(Q,\bfm)$ (see \cite[\S 1.4]{GLS17a}).

\subsection{Framed quiver representations with multiplicities}\label{Sect/framedQuivRep}

From the perspective of geometric invariant theory, it is useful to introduce a framing to eliminate a global stabiliser from the action and construct moduli spaces. In a representation theoretic context, the framing will correspond to the highest weight vector, and so also plays an important r\^{o}le.

\begin{definition}
An $\bff$-framed representation of $(Q,\bfm)$ is a pair $(M,s)$ consisting of a representation $M$ of $(Q,\bfm)$ and framing data $s$ consisting of $k_{m_i}$-linear maps $s_{i}\colon k_{m_i}^{\oplus f_i}\rightarrow M_i$ for each $i \in Q_0$. We say $(M,s)$ is locally free (resp.\ finitely generated) if $M$ is, and refer to $\rk M$ as the rank vector of $(M,s)$.

A morphism $h \colon (M,s)\rightarrow (M',s')$ of $\bff$-framed representations of $(Q,\bfm)$ is a morphism $h \colon M \rightarrow M'$ of $(Q,\bfm)$-representations that is compatible with the framing data in the sense that for all $i\in Q_0$, the following diagram commutes:
\[
\begin{tikzcd}
k_{m_i}^{\oplus f_i} \ar[r,"s_i"]\ar[d,equal] & M_{i} \ar[d,"h_i"] \\
k_{m_i}^{\oplus f_i} \ar[r,"s_i'"] & M'_{i}.
\end{tikzcd}
\]
\end{definition}

One can describe moduli of framed representations by adding an additional framing vertex and $f_i$ arrows from the framing vertex to each vertex $i$ as follows (see \cite{CB01}\cite[\S 2.3.2]{HJV25}).

\begin{definition}\label{Def quiver with mult associated to framing}
Given $(Q,\bfm)$ and framing vector $\bff$ and rank vector $\bfr \in \mathbb{N}^{Q_0}$, we define the associated framed quiver with multiplicities $(Q_\bff,\widehat{\bfm})$ and rank vector $\widehat{\bfr}$ for $Q_\bff$ by:
\begin{itemize}
\item $(Q_\bff)_0:=Q_0\sqcup\{\infty\}$ and  $(Q_\bff)_1:=Q_1\sqcup\bigsqcup_{i\in Q_0}\{b_{i,f} \colon \infty\rightarrow i,\ 1\leq f\leq f_i\}$,
\item $\widehat{m}_i:=
\left\{
\begin{array}{ll}
m_i & i\in Q_0, \\
1 & i=\infty.
\end{array}
\right.
$
\item $
\widehat{r}_i=
\left\{
\begin{array}{ll}
r_i & i\in Q_0, \\
1 & i=\infty.
\end{array}
\right.
$
\end{itemize}
\end{definition}

This paper primarily concerns the action of $\GL_{\bfm,\bfr}$ on $\mathrm{T}^*R({Q_\bff},\widehat{\bfm};\widehat{\bfr}) \cong R(\overline{Q_\bff},\widehat{\bfm};\widehat{\bfr})$, where $\overline{Q_\bff}$ is the doubled quiver obtained by adding an opposite arrow for each arrow in ${Q_\bff}$. We decompose this space as follows:
\[R(\overline{Q_\bff},\widehat{\bfm};\widehat{\bfr}) = R(Q,\bfm;\bfr) \times R(Q^{\mathrm{op}},\bfm;\bfr) \times \prod_{i \in Q_0} \! \! \left(\Hom_{k_{m_i}}(k^{\oplus f_i}_{m_i},k^{\oplus r_i}_{m_i}) \times \Hom_{k_{m_i}}(k^{\oplus r_i}_{m_i},k_{m_i}^{\oplus f_i}) \right) \]
and write points as $(x,y,(s_i,t_i)_{i \in Q_0})$. We also write $s = (s_i)_{i \in Q_0}$ and $t = (t_i)_{i \in Q_0}$.

\begin{notation}
Given a closed point $(x,y,s,t)\in R(\overline{Q_\bff},\widehat{\bfm};\widehat{\bfr})$ (resp.\ $(x,y)\in R(\overline{Q},\bfm;\bfr)$), we let $M(x,y,s,t)$ (resp.\ $M(x,y)$) denote the corresponding representation of $(\overline{Q_\bff},\widehat{\bfm})$ (resp.\ $(\overline{Q},\bfm)$).
\end{notation}

\subsection{Quivers with multiplicities as modulated graphs}

Let $(Q,\bfm)$ be a quiver with multiplicities. In this section, we recall how representations of quivers with multiplicities can be interpreted as representations of certain modulated graphs, following \cite[\S 5]{GLS17a}. We will use this language throughout the paper, in order to describe Hecke modifications of nilpotent quiver representations  of $(Q,\bfm)$ at a vertex (see \cref{Sect/CrystalComponents}).

Recall the bimodules $k_a,\ a\in Q_1$ from \cref{notation gcd of multiplicities}. Together with the rings $k_{m_i}$ for $ i\in Q_0$, these bimodules yield a modulation of the graph underlying $(Q,\bfm)$, as explained in \cite[\S 5]{GLS17a}.

\begin{lemma}\label{Lem/balancingBimodules}
\emph{\cite[Prop.\ 5.1]{GLS17a}} Let $a \colon i \rightarrow j$ be an arrow in $Q$. Given a $k_{m_i}$-module $M_i$ and a $k_{m_j}$-module $M_j$, there are functorial isomorphisms of $k$-vector spaces:
\[
\Hom_{k_{m_i}}(M_i,k_a\otimes_{k_{m_j}}M_j)
\simeq
\Hom_{k_{m_{ij}}}(M_i,M_j)
\simeq
\Hom_{k_{m_j}}(M_i\otimes_{k_{m_i}}k_a,M_j)
.
\]
\end{lemma}

The following maps are key to formulate modifications of representations of $(Q,\bfm)$ at a vertex in Geiss--Leclerc--Schr\"{o}er's work (known as reflection functors or Hecke modifications, see \cite{GLS17a,GLS18a}).

\begin{definition}\label{def/vertexKernelsCokernels}
Let $M$ be a representation of $(Q,\bfm)$ and $i\in Q_0$. Define the following maps:
\[
M_{i,\mathrm{out}} \colon M_i\overset{M_a}{\longrightarrow}\bigoplus_{a:i\rightarrow j}k_a\otimes_{k_{m_j}}M_j
,
\]
\[
M_{i,\mathrm{in}} \colon \bigoplus_{a \colon j\rightarrow i}M_j\otimes_{k_{m_j}}k_a\overset{M_a}{\longrightarrow} M_i
,
\]
which are induced from $M$ using \cref{Lem/balancingBimodules}.
We further define the $k_{m_i}$-modules:
\[
\sub_i(M):=\Ker(M_{i,\mathrm{out}}), \quad \fac_i(M):=\Coker(M_{i,\mathrm{in}})
.
\]
Finally, let $C_i(M)$ (resp.\ $K_i(M)$) denote the smallest quotient $M\twoheadrightarrow M''$ (resp.\ the smallest subrepresentation $M'\subset M$) such that $\Ker(M\twoheadrightarrow M'')$ (resp.\ $M/M'$) is concentrated at vertex $i$.
\end{definition}

\begin{notation}
We will sometimes denote by $\sub_i(M)$ (resp.\ $\fac_i(M)$) the representation $N$ of $(Q,\bfm)$ concentrated at vertex $i$, with $N_i=\sub_i(M)$ (resp.\ $N_i=\fac_i(M)$).
\end{notation}

The following lemma is straightforward. A similar statement can be found in \cite[\S 4.2]{GLS18a} for representations of preprojective algebras of quivers with multiplicities.

\begin{lemma}\label{Lem/propertiesK_iC_i}
Let $M$ be a representation of $(Q,\bfm)$ and $i\in Q_0$. Then there are short exact sequences of representations of $(Q,\bfm)$:
\[
0\rightarrow \sub_i(M)\rightarrow M\rightarrow C_i(M)\rightarrow 0
\]
\[
\emph{(resp.\ }
0\rightarrow K_i(M)\rightarrow M\rightarrow \fac_i(M)\rightarrow 0
\emph{).}
\]
\end{lemma}

Finally, the maps $M_{i,\mathrm{in}}$ and $M_{i,\mathrm{out}}$ can be used to rewrite moment map equations for quivers with multiplicities in a compact way. We will make use of this formulation in \cref{Sect/GeometricHWCrystals} to analyse the behaviour of stable representations of the preprojective algebra of $(Q,\bfm)$ under Hecke modifications. We first need to twist the map $M_{i,\mathrm{in}}$ by a sign.

\begin{definition}\label{def/signTwistedM_i,in}
Let $(\overline{Q},\bfm)$ be the doubled quiver associated to $(Q,\bfm)$. For $a\in Q_1$, let us define $\mathrm{sgn}(a):=1$ and $\mathrm{sgn}(a^*):=-1$. For $i\in Q_0$, we define the sign-twist of $M_{i,\mathrm{in}}$ by
\[
\widetilde{M}_{i,\mathrm{in}}\colon\bigoplus_{\substack{a:j\rightarrow i \\ a\in\overline{Q}_1}}M_j\otimes_{k_{m_j}}k_a\overset{\mathrm{sgn}(a)M_a}{\longrightarrow} M_i
.
\]
\end{definition}

\begin{lemma}\label{Lem/momentMapEqn}
\emph{\cite[Prop.\ 5.2]{GLS17a}} For $\bfr\in\bbZ_{\geq0}^{Q_0}$ and $(x,y)\in R(\overline{Q},\bfm;\bfr)$, let $M:=M(x,y)$. Then $\mu_{Q,\bfm;\bfr}(x,y)=0$ if and only if  we have $\widetilde{M}_{i,\mathrm{in}}\circ M_{i,\mathrm{out}}=0$ for every $i\in Q_0$.
\end{lemma}

\subsection{E-filtered representations and crystal components}\label{Sect/CrystalComponents}

Let $(Q,\bfm)$ be a quiver with multiplicities. From this point on, we assume that $Q$ has no loop arrows. In this section, we recall the notion of E-filtered quiver representations with multiplicities from \cite{GLS18a}. Those are the analogues of nilpotent quiver representations for quivers with multiplicities. We then recall the definition of Geiss--Leclerc--Schr\"{o}er's analogue $\mathrm{nil}_E(\Pi_{Q,\bfm})$ of Lusztig's nilpotent variety \cite[\S 12]{Lus91}. Unlike Lusztig's original variety, this one is not equidimensional and we recall the characterisation by Geiss--Leclerc--Schr\"{o}er of the top-dimensional irreducible components of $\mathrm{nil}_E(\Pi_{Q,\bfm})$ in terms of crystal representations of $(\overline{Q},\bfm)$ and how those components can be obtained by Hecke modification from the zero representation, following \cite[\S 4]{GLS18a}.

The following definition of E-filtered representations comes from \cite[\S 2.1]{GLS18a}.

\begin{definition}\label{def/E-filteredRep}
For $i\in Q_0$, let $E_i$ be the locally free representation of $(Q,\bfm)$ given by:
\[
(E_i)_j=
\left\{
\begin{array}{ll}
k_{m_i} & j=i \\
0 & j\ne i
\end{array}
\right.
.
\]
A (locally free) representation $M$ of $(Q,\bfm)$ is E-filtered if there exist subrepresentations
\[
0= M_n\subsetneq\ldots\subsetneq M_0=M,
\ n\in\bbZ_{\geq0}
\]
such that, for all $0\leq l\leq n-1$, there exists $i_l\in Q_0$ such that $M_l/M_{l+1}\simeq E_{i_l}$.
\end{definition}

An E-filtered representation is necessarily locally free, as it is an extension of locally free representations, namely of the representations $E_i$ for $i\in Q_0$. By construction, any E-filtered representation of $(Q,\bfm)$ can be obtained from a representation of smaller rank by Hecke modification.

\begin{lemma}\label{Lem/facSubE-filteredRep}
\emph{\cite[Lem.\ 3.1]{GLS18a}} Let $M$ be a non-zero E-filtered representation of $(Q,\bfm)$. Then there exists $i\in Q_0$ such that the $k_{m_i}$-module $\fac_i(M)$ (resp.\ $\sub_i(M)$) contains a non-zero free direct summand.
\end{lemma}

The following notion was introduced by Geiss--Leclerc--Schr\"{o}er in order to characterise top-dimensional components of $\mathrm{nil}_E(\Pi_{Q,\bfm})$ \cite[\S 4]{GLS18a}.

\begin{definition}\label{def/crystalRep}
Let $M$ be an $E$-filtered representation of $(Q,\bfm)$. We say that $M$ is crystal if $M\simeq 0$ or if, for every $i\in Q_0$, the $k_{m_i}$-modules $\sub_i(M)$ and $\fac_i(M)$ are free and, if non-isomorphic to $M$, the representations $K_i(M)$ and $C_i(M)$ are also crystal.
\end{definition}

Note that, given $i\in Q_0$, if $\sub_i(M)$ (resp.\ $\fac_i(M)$) is locally free, then $C_i(M)$ (resp.\ $K_i(M)$) is locally free and E-filtered. We now define the loci of E-filtered and crystal representations, and the analogue $\Lambda_{Q,\bfm}$ of Lusztig's nilpotent variety.

\begin{definition}\label{def/nilpotentCone}
Given $\bfr\in\bbZ_{\geq0}^{Q_0}$, we first define $\GL_{\bfm,\bfr}$-invariant constructible subsets as follows. Let
\[
\mathrm{nil}_{\bfr}:=\mathrm{nil}_E(\Pi_{Q,\bfm},\bfr)\subset\mu_{Q,\bfm;\bfr}^{-1}(0)
\]
be the subset whose points correspond to $E$-filtered representations.

Following \cite[\S 4.2]{GLS18a}, we say that an irreducible component $Y\subset\mathrm{nil}_{\bfr}$ is a crystal component if there is a non-empty open subset $Y'\subset Y$ such that, for all $(x,y)\in Y'$, the representation $M(x,y)$ is crystal.

Next, define the closed subscheme
\[
\Lambda_\bfr:=\Lambda_{Q,\bfm;\bfr}\subset\mu_{Q,\bfm;\bfr}^{-1}(0)
\]
to be the union of the closures (in $\mu_{Q,\bfm;\bfr}^{-1}(0)$) of crystal components of $\mathrm{nil}_{\bfr}$, with its reduced scheme structure. We call $\Lambda_\bfr^{\mathrm{cr}}\subset\Lambda_\bfr$ the constructible subset whose points correspond to crystal representations.

We call the irreducible components of $\Lambda_\bfr$ \emph{closed} crystal components. By construction, those are the closures of crystal components of $\mathrm{nil}_\bfr$. Finally, we will consider the following unions, the first of finite type schemes and the second of constructible subsets:
\[
\Lambda:=\bigsqcup_{\bfr\in\bbZ_{\geq0}^{Q_0}}\Lambda_\bfr
\text{  and  }
\Lambda^{\mathrm{cr}}:=\bigsqcup_{\bfr\in\bbZ_{\geq0}^{Q_0}}\Lambda_\bfr^{\mathrm{cr}}
.
\]
\end{definition}

Note that all orbits corresponding to crystal representations are contained in $\Lambda$ (see \cref{Rmk/crystalLocusVScrystalComponents} below). Constructibility can be checked by viewing $\mathrm{nil}_E(\Pi_{Q,\bfm},\bfr)$ as the image of the projection from a quiver flag variety with multiplicities (similarly to \cite[\S 1.5]{Lus91}) and noting that for $i\in Q_0$, the condition that $\sub_i(M)$ (resp.\ $\fac_i(M)$) is free over $k_{m_i}$ is constructible (for instance, it is equivalent to $\dim_k\left(\sub_i(M)\right)=m_i\rk(\epsilon^{m_i-1}\bullet)$).

\begin{notation}
Throughout the paper, given an irreducible scheme $Y$, we will say that a general point in $Y$ satisfies a certain property (P) if there exists a non-empty (hence dense) open subset $Y'\subset Y$ such that all points $z\in Y'$ satisfy property (P).
\end{notation}

An important result by Geiss--Leclerc--Schr\"{o}er asserts that crystal components are exactly the top-dimensional components of $\mathrm{nil}_E(\Pi_{Q,\bfm},\bfr)$.

\begin{proposition}\label{Prop/dimCrystalComponents}
\emph{\cite[Thm.\ 4.1, Prop.\ 4.4]{GLS18a}}
Let $\bfr\in\bbZ_{\geq0}^{Q_0}$ and $Y\subset \mathrm{nil}_{\bfr}$ be an irreducible component. Then:
\[
\dim Y\leq\dim R(Q,\bfm;\bfr)=\frac{1}{2}\dim R(\overline{Q},\bfm;\bfr)
\]
and equality holds if and only if $Y$ is a crystal component.
\end{proposition}

\begin{example}\label{Exmp/NonMaxComponentsTypeC2}\cite[\S 8.2.2]{GLS18a}
Let us consider the quiver with multiplicities $(Q,\bfm)$, where
\[
Q=\bullet_1\longrightarrow\bullet_2
\]
and $\bfm=(m_1,m_2)=(1,2)$. For $\bfr\in\bbZ_{\geq0}^{Q_0}$, we have a $\GL_{\bfm,\bfr}$-equivariant isomorphism:
\[
R(\overline{Q},\bfm;\bfr)\simeq\mathrm{Mat}_{r_2\times r_1}(k_2)\times\mathrm{Mat}_{r_1\times r_2}(k_2)
,
\]
where $\GL_{\bfm,\bfr}$ acts on the right-hand side by left-right multiplication. Under this isomorphism, for $(A,B)\in\mathrm{Mat}_{r_2\times r_1}(k_2)\times\mathrm{Mat}_{r_1\times r_2}(k_2)$, we have:
\[
\mu_{(Q,\bfm),\bfr}(A,B)=(-BA[\epsilon],AB)
,
\]
where $BA[\epsilon]$ denotes the matrix obtained by extracting the coefficients of $\epsilon$ from $BA\in\mathrm{Mat}_{r_1\times r_1}(k_2)$. From this description, we can compute the irreducible components of $\mathrm{nil}_E(\Pi_{Q,\bfm},\bfr)$ for small values of $\bfr$.

Fix $\bfr=(1,2)$. Then the action $\GL_{\bfm,\bfr}\curvearrowright R(Q,\bfm;\bfr)$ admits three orbits, represented respectively by $(\begin{smallmatrix}1 \\ 0\end{smallmatrix})$, $(\begin{smallmatrix}\epsilon \\ 0\end{smallmatrix})$ and $0$. The projection $\mu_{(Q,\bfm),\bfr}^{-1}(0)\rightarrow R(Q,\bfm;\bfr)$ restricts to a vector bundle over each of those orbits (its conormal bundle, see \cite[Prop.\ 4.1.3.ii]{Gin12}), with respective fibres $\{0\}$, $\epsilon\mathrm{Mat}_{r_1\times r_2}(k_2)$ (at $(\begin{smallmatrix}\epsilon \\ 0\end{smallmatrix})$) and $R(Q,\bfm;\bfr)^\vee$. The conormal bundles of the first and third orbits consist of E-filtered representations. Moreover, the subset of E-filtered representations in the conormal bundle of $\GL_{\bfm,\bfr}\cdot (\begin{smallmatrix}\epsilon \\ 0\end{smallmatrix})$ is the sub-bundle formed by the orbits of $\left((\begin{smallmatrix}\epsilon \\ 0\end{smallmatrix}),(\begin{smallmatrix}0 & \epsilon\end{smallmatrix})\right)$ and $\left((\begin{smallmatrix}\epsilon \\ 0\end{smallmatrix}),0\right)$. Note that the corresponding representations are not crystal, as in both cases $\fac_2$ is not free over $k_2$.

Therefore, we obtain two closed crystal components, $R(Q,\bfm;\bfr)\subset\Lambda_\bfr$ and $R(Q,\bfm;\bfr)^\vee\subset\Lambda_\bfr$, and one non-crystal component, the orbit closure of $\left((\begin{smallmatrix}\epsilon \\ 0\end{smallmatrix}),(\begin{smallmatrix}0 & \epsilon\end{smallmatrix})\right)$. A direct computation shows that its dimension is $3<4=\dim R(Q,\bfm;\bfr)$. This is consistent with \cite[\S 8.2.2]{GLS18a} (note that the dimensions computed in \emph{loc.\ cit.\ } differ from ours by $8$, since quiver moduli are presented in a different, although equivalent fashion, see \cite[\S 2.2-2.3]{GLS18a}).
\end{example}

As a consequence of \cref{Prop/dimCrystalComponents}, we can show that $\Lambda_\bfr\subset\mathrm{T}^*R(Q,\bfm;\bfr)$ is a lagrangian subscheme. This generalises results of Lusztig's to quivers with multiplicities \cite{Lus91,Lus92}. Our proof follows the argument given in \cite[Thm.\ 4.12]{S12b}.

\begin{proposition}\label{Prop/lagrangianNilCone}
Suppose $k=\bbC$ and let $\bfr\in\bbZ_{\geq0}^{Q_0}$. Then $\Lambda_\bfr\subset\mathrm{T}^*R(Q,\bfm;\bfr)$ is a  lagrangian subscheme in the sense of \emph{\cite[Def.\ 1.3.24]{CG97}}.
\end{proposition}

\begin{proof}
Since, by \cref{Prop/dimCrystalComponents}, $\Lambda_\bfr$ has pure dimension $\frac{1}{2}\mathrm{T}^*R(Q,\bfm;\bfr)$, it remains to check that it is an isotropic subscheme in the sense of \cite[Def.\ 1.3.24]{CG97}. To do this, we show that $\mathrm{nil}_\bfr$ lies in the image of an isotropic subscheme under a well-chosen lagrangian correspondence. Applying \cite[Prop.\ 1.3.30, Prop.\ 2.7.51]{CG97}, we obtain that $\mathrm{nil}_\bfr$ is isotropic. It follows that $\Lambda_\bfr$ is isotropic, as its tangent spaces at smooth points are limits of tangent spaces of $\mathrm{nil}_\bfr\cap\Lambda_\bfr$ at smooth points in the Grassmann-bundle of $\mathrm{T}\left(R(\overline{Q},\bfm;\bfr)\right)$ and as isotropic subspaces form a closed subset thereof.

Let us write for short $G:=\GL_{\bfm,\bfr}$, $\mathfrak{g}:=\mathfrak{gl}_{\bfm,\bfr}$, $X:=R(Q,\bfm;\bfr)$, $Y:=\bigsqcup_{\underline{i}}G/P_{\underline{i}}$ where the disjoint union runs over the set of sequences $\underline{i}=\{i_1,\ldots,i_{\vert\bfr\vert}\}$, where $\vert\bfr\vert:=\sum_{i\in Q_0}r_i$ and $P_{\underline{i}}$ is the stabiliser of a fixed flag of locally free submodules
\[
N_{\underline{i},\bullet}\subset\bigoplus_{i\in Q_0}k_{m_i}^{\oplus r_i}
\]
of shape $\underline{i}$. Thus $Y$ parametrises flags $N_\bullet=(N_0\supset N_1\supset\ldots)$ of locally free submodules with subquotients of rank $1$, of all possible shapes. Let $Z\subset X\times Y$ be the subscheme parametrising pairs $(x,N_\bullet)\in X\times Y$ such that $x(N_\bullet)\subset N_{\bullet+1}$. Note that the projection $Z\rightarrow Y$ is a vector bundle, so $Z$ is smooth. Then we claim that:
\[
\mathrm{T}_Z^*(X\times Y)
\simeq
\left\{
(x,y,N_\bullet,z)\in\mathrm{T}^*X\times Y\times\mathfrak{g}
\left\vert
\begin{array}{l}
x(N_\bullet)\subset N_{\bullet+1}, \\
y(N_\bullet)\subset N_\bullet, \\
z=-\mu_{Q,\bfm;\bfr}(x,y).
\end{array}
\right.
\right\}
.
\]
Then the projection of $\mathrm{T}_Z^*(X\times Y)\cap (\mathrm{T}^*X\times Y)$ on $\mathrm{T}^*X$ contains $\mathrm{nil}_\bfr$. Since $Z\subset\mathrm{T}^*Z$ and $(\mathrm{T}^*X\times Y)\vert_Z=\mathrm{T}^*X\vert_Z\subset\mathrm{T}^*Z\times\mathrm{T}^*X$ are isotropic, this projection also is, by \cite[Prop.\ 2.7.51]{CG97} (applied to $M_1=\mathrm{pt}$, $M_2=\mathrm{T}^*Z$, $M_3=\mathrm{T}^*X$, $Z_{12}=Z$ and $Z_{23}=(\mathrm{T}^*X\times Y)\vert_Z$ in the notation of \emph{loc.\ cit.\ }). Thus $\mathrm{nil}_\bfr$ is also isotropic, by \cite[Prop. 1.3.30]{CG97}.

Let us now prove the claim. As in \cite[\S 13.4]{Lus91}, the tangent space of $Z$ at $(x,g)\in X\times G/P_{\underline{i}}$ is identified with the kernel of the following map:
\[
\begin{array}{rcl}
X\times\mathfrak{g}/g\cdot\mathfrak{p}_{\underline{i}} & \rightarrow & X/g\cdot X_{\underline{i}}\\
(x',\xi) & \mapsto & -\xi\cdot x+x'
,
\end{array}
\]
where $X_{\underline{i}}\subset X$ denotes the affine subspace defined by the condition $x(N_{\underline{i},\bullet})\subset N_{\underline{i},\bullet+1}$. Therefore, the conormal space of $Z$ at $(x,g)$ is the subspace of $X^\vee\times g\cdot\mathfrak{p}_{\underline{i}}^\perp$ spanned by vectors of the form $(y,-\mu_{Q,\bfm;\bfr}(x,y))$, where $y\in (g\cdot X_{\underline{i}})^\perp$.
\end{proof}

\begin{remark}
The proof of \cite[Thm.\ 4.12]{S12b} relies on \cite[Prop.\ 8.3.11]{KS90} instead of \cite[Prop.\ 2.7.51]{CG97} and assumes properness of $Y$ (in the case of quivers without multiplicities). That assumption is not necessary, since, in \cite[Prop.\ 8.3.11]{KS90}, properness is only assumed so that lagrangian correspondences preserve subanalytic subsets, but is not used in the proof of isotropy.
\end{remark}

In order to understand how crystal components are permuted under the operators of Hecke modification, it is necessary to examine certain loci where $\fac_i(M)$ is of fixed type.

\begin{definition}\label{def/cokernelLoci}
Let $\bfr\in\bbZ_{\geq0}^{Q_0}$, $c\in\bbZ_{\geq0}$ and $i\in Q_0$. We define
\[
W_{\bfr,c\bfalpha_i} \subset \mu^{-1}_{Q,\bfm;\bfr}(0)
\]
as the locally closed subscheme whose points correspond to representations $M$ satisfying $\fac_i(M)\simeq k_{m_i}^{\oplus c}$. Likewise, we define the $\GL_{\bfm,\bfr}$-invariant constructible subset $\mathrm{nil}_{\bfr,c\bfalpha_i}:=\mathrm{nil}_\bfr\cap W_{\bfr,c\bfalpha_i}$. Finally, we define the locally closed subscheme $\Lambda_{\bfr,c\bfalpha_i}\subset\Lambda_\bfr$ as the union of irreducible components of $\Lambda_\bfr\cap W_{\bfr,c\bfalpha_i}$ of dimension $\dim\Lambda_\bfr$.\end{definition}

\begin{remark}\label{Rem/localClosednessCokernelLoci}
Local closedness of $W_{\bfr,c\bfalpha_i}$ can be checked by viewing $W_{\bfr,c\bfalpha_i} \subset \mu^{-1}_{Q,\bfm;\bfr}(0)$ as an open subset of the locally closed subset where $\dim_k\fac_i(M)=m_ic$, that is the subset where moreover $\rk(\epsilon^{m_i-1}\bullet\vert\fac_i(M))\geq c$.
\end{remark}

\begin{notation}\label{Notation/0alpha_i}
When $c=0$, we will still write $\mathrm{nil}_{\bfr,0\bfalpha_i}$ (resp.\ $\Lambda_{\bfr,0\bfalpha_i}$) in order to keep track of the vertex $i$ where $\fac_i(M)=0$.
\end{notation}

Note that we have $\Lambda_{\bfr,c\bfalpha_i}\subset W_{\bfr,c\bfalpha_i}\cap\Lambda_\bfr$ and the inclusion can be strict.

\begin{example}\label{Exmp/cokernelLociCrystalComponents}
Fix $(Q,\bfm)$ as in \cref{Exmp/NonMaxComponentsTypeC2} and $\bfr=(2,2)$. Then there are finitely many orbits for the action $\GL_{\bfm,\bfr}\curvearrowright R(Q,\bfm;\bfr)$, represented respectively by:
\[
A_1=(\begin{smallmatrix}1 & 0 \\ 0 & 1\end{smallmatrix}), \  A_2=(\begin{smallmatrix}1 & 0 \\ 0 & \epsilon\end{smallmatrix}), \ 
A_3=(\begin{smallmatrix}1 & \epsilon \\ 0 & 0\end{smallmatrix}), \ 
A_4=(\begin{smallmatrix}1 & 0 \\ 0 & 0\end{smallmatrix}), \ 
A_5=(\begin{smallmatrix}\epsilon & 0 \\ 0 & \epsilon\end{smallmatrix}), \ 
A_6=(\begin{smallmatrix}\epsilon & 0 \\ 0 & 0\end{smallmatrix})\text{ and } 
A_7=0.
\]
In order to describe the conormal spaces, we let
\[  B_3:=(\begin{smallmatrix}0 & -\epsilon \\ 0 & 1\end{smallmatrix}),\ B_4:=(\begin{smallmatrix}0 & 0 \\ 1 & 0\end{smallmatrix}),\  B'_4:=(\begin{smallmatrix}0 & 0 \\ 0 & 1\end{smallmatrix})\text{ and } 
B_7 := A_1.\]
Then the conormal space $C_i$ at $A_i$, when non-zero and generically crystal, is given by:
\[
C_3=
\left\{
yB_3
\ \vert\ y\in k_2
\right\}
,\ 
C_4=
\left\{
y_1B_4+y_2B'_4\ \vert\ y_1\in k,\ y_2\in k_2
\right\}
,
\]
and $C_7=R(Q,\bfm;\bfr)^\vee$.
Therefore, the closed crystal components are the respective orbit closures of $(A_1,0)$, $(A_3,B_3)$, $(A_4,B_4)$ and $(0,B_7)$. Let us now consider:
\[
\Lambda_3=\overline{\GL_{\bfm,\bfr}\cdot (A_3,B_3)}\subset\Lambda_\bfr.
\]
Then $\Lambda_3\cap W_{\bfr,\bfalpha_1}$ contains the orbit of $(A_3,\epsilon B_3)$ as a dense subset, which is not contained in any other closed crystal component and whose dimension is smaller than $\dim\Lambda_3=\dim\Lambda_\bfr$. Therefore, we have found an irreducible component of $\Lambda_\bfr\cap W_{\bfr,\bfalpha_1}$ which is not top-dimensional, hence not contained in $\Lambda_{\bfr,\bfalpha_1}$.
\end{example}

Over $\Lambda_{\bfr,c\bfalpha_i}$, the spaces of Hecke modifications, which consist of nested representations $M'\subset M$ of $\Pi_{Q,\bfm}$ satisfying $M/M'\simeq E_i$ for some $i\in Q_0$, are geometrically well-behaved.

\begin{definition}\label{def/quiverHeckeCorrespondence}
Let $\bfr\in\bbZ_{\geq0}^{Q_0}$, $c,l\in\bbZ_{\geq0}$ and $i\in Q_0$. Then define the closed subscheme
\[
W_{\bfr,\bfr+l\bfalpha_i;(c+l)\bfalpha_i}
\subseteq 
W_{\bfr+l\bfalpha_i,(c+l)\bfalpha_i}\times\Hom_{k_{m_i}}^{\mathrm{inj}}(k_{m_i}^{\oplus r_i},k_{m_i}^{\oplus (r_i+l)})
\] of tuples $((x,y),\varphi)$ such that $\Ima(M(x,y)_{i,\mathrm{in}})\subset\Ima(\varphi)$. This comes with maps  
\[
W_{\bfr,c\bfalpha_i}
\overset{q}{\longleftarrow}
W_{\bfr,\bfr+l\bfalpha_i;(c+l)\bfalpha_i}
\overset{p}{\longrightarrow}
W_{\bfr+l\bfalpha_i,(c+l)\bfalpha_i}
,
\]
where $p$ is the projection $((x,y),\varphi)\mapsto (x,y)$ and $q$ maps $((x,y),\varphi)$ to the restriction of $(x,y)$ along $\varphi$. We refer to such diagrams as Hecke correspondences.
\end{definition}

\begin{notation}
Fix $\bfr\in\bbZ_{\geq0}^{Q_0}$, $c,l\in\bbZ_{\geq0}$ and $i\in Q_0$ as above. We will use the notations:
\begin{align*}
\mathrm{nil}_{\bfr,\bfr+l\bfalpha_i;(c+l)\bfalpha_i} & :=p^{-1}(\mathrm{nil}_{\bfr+l\bfalpha_i,(c+l)\bfalpha_i}), \\
\Lambda_{\bfr,\bfr+l\bfalpha_i;(c+l)\bfalpha_i} & :=p^{-1}(\Lambda_{\bfr+l\bfalpha_i,(c+l)\bfalpha_i}).
\end{align*}
\end{notation}

\begin{remark}\label{Rem/restrictionHeckeCorrespondence}
By construction, the maps
\[
W_{\bfr,c\bfalpha_i}
\overset{q}{\longleftarrow}
W_{\bfr,\bfr+l\bfalpha_i;(c+l)\bfalpha_i}
\overset{p}{\longrightarrow}
W_{\bfr+l\bfalpha_i,(c+l)\bfalpha_i}
\]
satisfy $p^{-1}(\mathrm{nil}_{\bfr+l\bfalpha_i,(c+l)\bfalpha_i})=q^{-1}(\mathrm{nil}_{\bfr,c\bfalpha_i})$. Moreover, by \cite[Lem.\ 3.4-3.6]{GLS18a}, $p$ and $q$ are smooth, surjective, with connected fibres and satisfy:
\[
\dim(q)-\dim(p)=\dim(R(Q,\bfm;\bfr+l\bfalpha_i))-\dim(R(Q,\bfm;\bfr))
.
\]
Note that \cite[Lem.\ 3.4-3.6]{GLS18a} is stated for $\mathrm{nil}_{\bfr,c\bfalpha_i},\mathrm{nil}_{\bfr+l\bfalpha_i,(c+l)\bfalpha_i}$, but the proof makes no use of the E-filtered property, so the result also holds for $W_{\bfr,c\bfalpha_i},W_{\bfr+l\bfalpha_i,(c+l)\bfalpha_i}$. Therefore, the restricted maps 
\[
\mathrm{nil}_{\bfr,c\bfalpha_i}
\overset{q}{\longleftarrow}
\mathrm{nil}_{\bfr,\bfr+l\bfalpha_i;(c+l)\bfalpha_i}
\overset{p}{\longrightarrow}
\mathrm{nil}_{\bfr+l\bfalpha_i,(c+l)\bfalpha_i}
\]
induce a one-to-one correspondence
\[
\begin{array}{rcl}
\mathrm{Irr}(\mathrm{nil}_{\bfr,c\bfalpha_i}) & \longleftrightarrow & \mathrm{Irr}(\mathrm{nil}_{\bfr+l\bfalpha_i,(c+l)\bfalpha_i}) \\
Y'=q(p^{-1}(Y)) & \longleftrightarrow & Y=p(q^{-1}(Y'))
\end{array}
\]
such that $\dim Y=\dim Y'+\dim(q)-\dim(p)$. From \cref{Prop/dimCrystalComponents}, we obtain that the original maps $p$ and $q$ restrict as follows to smooth surjective maps with connected fibres
\[
\Lambda_{\bfr,c\bfalpha_i}
\overset{q}{\longleftarrow}
\Lambda_{\bfr,\bfr+l\bfalpha_i;(c+l)\bfalpha_i}
\overset{p}{\longrightarrow}
\Lambda_{\bfr+l\bfalpha_i,(c+l)\bfalpha_i}
\]
and induce a similar one-to-one correspondence between $\mathrm{Irr}(\Lambda_{\bfr,c\bfalpha_i})$ and $\mathrm{Irr}(\Lambda_{\bfr+l\bfalpha_i,(c+l)\bfalpha_i})$. Indeed, given $Y'\in\mathrm{Irr}(\mathrm{nil}_{\bfr,c\bfalpha_i})$ and $Y=p(q^{-1}(Y'))\in\mathrm{Irr}(\mathrm{nil}_{\bfr+l\bfalpha_i,(c+l)\bfalpha_i})$, we have:
\[
q^{-1}(\overline{Y'})=\overline{q^{-1}(Y')}=\overline{p^{-1}(Y)}=p^{-1}(\overline{Y}),
\]
where closures are taken in $W_{\bfr,\bfr+l\bfalpha_i;(c+l)\bfalpha_i}$. The inclusions $p^{-1}(\overline{Y})\subset \overline{p^{-1}(Y)}$ and $\overline{q^{-1}(Y')}\subset q^{-1}(\overline{Y'})$ hold because $p$ and $q$ are open. Since $p$ and $q$ are also surjective, we obtain $\overline{Y}=p(q^{-1}(\overline{Y'}))$ and $\overline{Y'}=q(p^{-1}(\overline{Y}))$ as required.
\end{remark}

\begin{remark}\label{Rmk/crystalLocusVScrystalComponents}
The above observations imply that all points $(x,y)\in\mathrm{nil}_\bfr$ such that $M(x,y)$ is crystal lie in $\Lambda_\bfr$. In fact, the crystal locus is contained in the union of crystal components.

Indeed, suppose that $M:=M(x,y)$ is crystal for a given $(x,y)\in\mathrm{nil}_\bfr$. Then $(x,y)\in\mathrm{nil}_{\bfr,c\bfalpha_i}$ for some $i\in Q_0$ and $c=\rk\fac_i(M)>0$. With the notation of \cref{Rem/restrictionHeckeCorrespondence}, there exists $(x',y')\in\mathrm{nil}_{\bfr-c\bfalpha_i,0\bfalpha_i}$ such that $(x,y)\in p(q^{-1}(x',y'))$. Since $M':=M(x',y')$ is also crystal, we may assume that $(x',y')$ is contained in a crystal component $Y'\subset\mathrm{nil}_{\bfr-c\bfalpha_i}$, by induction on $\bfr$. Since $\fac_{i}(M')=0$ is an open condition, we obtain that $Y'\cap\mathrm{nil}_{\bfr-c\bfalpha_i,0\bfalpha_i}\subset Y'$ is open (and dense). Therefore $Y:=\overline{p(q^{-1}(Y'\cap\mathrm{nil}_{\bfr-c\bfalpha_i,0\bfalpha_i}))}$ (closure taken in $\mathrm{nil}_\bfr$) has dimension $\dim R(Q,\bfm;\bfr)$, hence $Y\subset\mathrm{nil}_\bfr$ is a crystal component by \cref{Prop/dimCrystalComponents}, and $(x,y)\in Y$.
\end{remark}

\subsection{General filtrations of E-filtered representations}

In this section, we define a dense locus in $\Lambda$, where representations of $\Pi_{Q,\bfm}$ admit a well-behaved filtration with subquotients of the form $E_i$ for $ i\in Q_0$. This will be key in our construction of nilpotent quiver varieties with multiplicities (see \cref{Sect/RecursiveQuotient}). We encode the combinatorics of such filtrations by a \lq type' as follows. Recall that we assume $Q$ has no loop arrows.

\begin{definition}\label{def/filtrationType}
A type is an ordered tuple $F=(c_1\bfalpha_{i_1},\ldots,c_s\bfalpha_{i_s})$, where for $1\leq t\leq s$, we have $c_t\in\bbZ_{>0}$ and $i_t\in Q_0$. The associated rank vector of $F$ is $\bfr=\sum_{t=1}^sc_t\bfalpha_{i_t}$. We say a representation $M$ of $(Q,\bfm)$ of rank $\bfr$ is of type $F$ if 
\begin{enumerate}[nosep, label=\roman*)]
    \item $\fac_{i_1}(M)\simeq k_{m_{i_1}}^{\oplus c_1}$ and
    \item $K_{i_1}(M)$ has type $F':=(c_2\bfalpha_{i_2},\ldots,c_s\bfalpha_{i_s})$.
\end{enumerate}
Given a type $F$ with associated rank vector $\bfr$, we call $\mathrm{nil}_F\subset\mu_{Q,\bfm;\bfr}^{-1}(0)$ the (constructible) subset of those points corresponding to representations of type $F$.
\end{definition}

Constructibility of  $\mathrm{nil}_F\subset\mu_{Q,\bfm;\bfr}^{-1}(0)$ can be checked inductively using the Hecke correspondences discussed in \cref{Rem/restrictionHeckeCorrespondence}.

Note that, for any $F$, all representations of type $F$ are $E$-filtered. A representation $M$ of rank $\bfr$ may be of type $F$ for several tuples $F$, or none. If $M$ is crystal, then there exists at least one such tuple. Furthermore, in a given crystal component $Y$, there is at least one such tuple which works for general representations parametrised by $Y$.

\begin{lemma}\label{Lem/typeCrystalComponent}
Fix a type $F=(c_1\bfalpha_{i_1},\ldots,c_s\bfalpha_{i_s})$ of rank $\bfr$. Then there is at most one closed crystal component containing $\mathrm{nil}_F$ as a dense subset. Conversely, for any closed crystal component $Z\subset\Lambda_\bfr$, there is at least one type $F$ such that $\mathrm{nil}_F\cap Z\subset Z$ is dense and, for any such type, $\mathrm{nil}_F$ is a dense open subset of $Z$.
\end{lemma}

\begin{proof}
If $\mathrm{nil}_F$ is a dense subset of a closed crystal component $Z\subset\Lambda_\bfr$, then $Z=\overline{\mathrm{nil}_F}$, so $Z$ is uniquely determined. Conversely, given a closed crystal component $Z\subset\Lambda_\bfr$, the dense subset $\Lambda_\bfr^{\mathrm{cr}}\cap Z\subset Z$ is covered by the constructible subsets $\mathrm{nil}_F\cap Z$, where $F$ runs over the finite set of types with associated rank vector $\bfr$. So there exists at least one type $F$ such that $\mathrm{nil}_F\cap Z\subset Z$ is dense.

For $1\leq t\leq s$, let us call $F_t:=(c_t\bfalpha_{i_t},\ldots,c_s\bfalpha_{i_s})$. We now show by induction on $s\geq1$ that, for all $1\leq t\leq s$, whenever $\mathrm{nil}_{F_t}\cap Z\subset Z$ is dense for a certain closed crystal component $Z\subset\Lambda$, then $\mathrm{nil}_{F_t}$ is moreover a dense open subset of $Z$. If $s=1$, then $\mathrm{nil}_{c_1\bfalpha_{i_1}}=\mu_{Q,\bfm;c_1\bfalpha_{i_1}}^{-1}(0)=\mathrm{pt}$ (recall that $Q$ has no loop arrows), so all claims follow immediately. On the other hand, if $s>1$, we write $F=(c\bfalpha_i,F')$ for $\bfr'=\bfr - c\bfalpha_i$. The induction hypothesis for $F'$ implies that the claim holds for $2\leq t\leq s$. For $t=1$, we consider the Hecke correspondences (see \cref{Rem/restrictionHeckeCorrespondence}):
\[
\mathrm{nil}_{\bfr',0\bfalpha_i}
\overset{q}{\longleftarrow}
\mathrm{nil}_{\bfr',\bfr;c\bfalpha_i}
\overset{p}{\longrightarrow}
\mathrm{nil}_{\bfr,c\bfalpha_i},
\]
\[
\Lambda_{\bfr',0\bfalpha_i}
\overset{\overline{q}}{\longleftarrow}
\Lambda_{\bfr',\bfr;c\bfalpha_i}
\overset{\overline{p}}{\longrightarrow}
\Lambda_{\bfr,c\bfalpha_i}.
\]
Then we have:
\[
p^{-1}(\mathrm{nil}_F)=q^{-1}(\mathrm{nil}_{F'}\cap\mathrm{nil}_{\bfr',0\bfalpha_i})
.
\]
Recall we assumed that $\mathrm{nil}_F\cap Z\subset Z$ is dense. Let us also consider the closed crystal component $Z':=\overline{q(p^{-1}(Z\cap\mathrm{nil}_{\bfr,c\bfalpha_i}))}\subset\Lambda$. Then by construction, $\mathrm{nil}_{F'}\cap Z'\subset Z'$ is dense and we have by induction that $\mathrm{nil}_{F'}$ is an open subset of $Z'$. Thus $\mathrm{nil}_F$ is contained in $\Lambda$ and we have:
\[
\overline{p}^{-1}(\mathrm{nil}_F)=\overline{q}^{-1}(\mathrm{nil}_{F'}\cap\Lambda_{\bfr',0\bfalpha_i}).
\]
Since $\overline{p}$ is open and surjective, we also deduce that $\mathrm{nil}_F$ is an open subset of $Z\cap\Lambda_{\bfr,c\bfalpha_i}$. Moreover, $Z\cap\Lambda_{\bfr,c\bfalpha_i}\subset Z$ is dense and open by construction, so $\mathrm{nil}_F$ is an open subset of $Z$ itself, as required.
\end{proof}

We finish with the definitions of the various loci we will work with during the construction of nilpotent quiver varieties with multiplicities.

\begin{definition}\label{def/typeCrystalComponent} 
Let $B:=\mathrm{Irr}(\Lambda)$. For $b\in B$, let us denote by $Z_b$ the corresponding irreducible component. Let us also denote by $\mathcal{F}$ the set of types $F$ such that $\mathrm{nil}_F$ is dense in some closed crystal component. For $F\in\mathcal{F}$, let $b_F\in B$ correspond to the unique closed crystal component containing $\mathrm{nil}_F$. We will use the alternative notations
\begin{align*}
& \Lambda_F :=\mathrm{nil}_F\subset\Lambda \\
& Z_{b_F,F} :=\mathrm{nil}_F\subset Z_{b_F}
\end{align*}
to emphasise the ambient space. We further define $\Lambda_{\mathrm{filt}}$  and $Z_{b,\mathrm{filt}}$ for $b \in B$ as follows:
\begin{equation}\label{eqn stratification of Lambda filt}
\Lambda_{\mathrm{filt}}:=\bigcup_{F\in\mathcal{F}}\Lambda_F
\subset
\Lambda
\text{ ; } \quad 
Z_{b,\mathrm{filt}}:=\bigcup_{b_F=b}Z_{b,F}\subset Z.
\end{equation}
\end{definition}

\begin{remark}
The following phenomenon frequently occurs: there exist distinct closed crystal components $Z,Z'\subset\Lambda_\bfr$, which respectively admit generic types $F,F'$ and such that: 
\[
\emptyset\ne Z'_{F'}\cap Z\subset Z\setminus Z_F
.
\]
We will later construct explicit Zariski-locally trivial quotients of stable loci in $Z_F,Z'_{F'}$ under the action of $\GL_{\bfm,\bfr}$ (see \cref{Sect/constructionNilpQuiverVar}), when $Q$ is a framed quiver. The above phenomenon shows that neither $Z_F$ nor $Z'_{F'}$ is an open neighbourhood of $Z'_{F'}\cap Z$ in $\Lambda$. It is therefore impractical to glue the various quotients $Z_F^\st/\GL_{\bfm,\bfr}$ into a quotient of some subset of $\Lambda_\bfr$.

More precisely, for $(Q,\bfm)$ as in \cref{Exmp/NonMaxComponentsTypeC2} and $\bfr=(2,2)$, consider the framed quiver $(Q_\bff,\widehat{\bfm})$, where $\bff=(2,2)$. A typical example of the above situation arises when
\[
Z'=Z'_{F'}=R(Q,\bfm;\bfr)\times\prod_{i\in Q_0}\Hom_{k_{m_i}}(k_{m_i}^{\oplus r_i},k_{m_i}^{\oplus f_i})
\subset\Lambda_{\widehat{\bfr}}
,
\]
where $F'=(2\bfalpha_1,2\bfalpha_2,\bfalpha_\infty)$ and $Z$ is any closed crystal component for which $F'$ is not a generic type, e.g. $Z$ is the closure of
\[
\GL_{\bfm,\bfr}\cdot (A_3,B_3)\times\prod_{i\in Q_0}\Hom_{k_{m_i}}(k_{m_i}^{\oplus r_i},k_{m_i}^{\oplus f_i})
\]
where $(A_3,B_3)$ are as in \cref{Exmp/cokernelLociCrystalComponents}. In this example, we have:
\begin{align*}
& R(Q,\bfm;\bfr)\times\prod_{i\in Q_0}\Hom_{k_{m_i}}^{\mathrm{inj}}(k_{m_i}^{\oplus r_i},k_{m_i}^{\oplus f_i})\subset(Z'_{F'})^\st
, \\
& \GL_{\bfm,\bfr}\cdot (A_3,B_3)\times\prod_{i\in Q_0}\Hom_{k_{m_i}}^{\mathrm{inj}}(k_{m_i}^{\oplus r_i},k_{m_i}^{\oplus f_i})\subset Z_F^\st
,
\end{align*}
so that $Z'_{F'} \cap Z$ meets the stable locus (see \cref{def/stable}).
\end{remark}

\section{Construction of nilpotent quiver varieties}\label{Sect/constructionNilpQuiverVar}

In this section, we introduce stability conditions for framed quiver representations with multiplicities and construct analogues of Nakajima's nilpotent quiver varieties, as in \cref{MainThm/quotientConstruction}. As above, we assume that $Q$ has no loop arrows.

\subsection{Stability conditions}

In this section, we introduce the stability condition which we will use to construct nilpotent quiver varieties for framed quivers with multiplicities. We also collect a number of properties of stable, framed, E-filtered quiver representations with multiplicities, which we will use throughout our quotient construction in the following sections. Let $(Q,\bfm)$ be a quiver with multiplicities. Let $\bfr\in\bbZ_{\geq0}^{Q_0}$ be a rank vector and $\bff\in\bbZ_{\geq0}^{Q_0}$ a framing vector with associated framed quiver with multiplicities $(Q_\bff,\widehat{\bfm})$. 

\begin{definition}\label{def/stable}
Let $M$ be a locally free, rank $\widehat{\bfr}$ representation of $(\overline{Q_\bff},\widehat{\bfm})$. Let $t_i:M_i\rightarrow k_{m_i}^{\oplus f_i}$ be the map obtained by collecting the $f_i$ arrows $i\rightarrow\infty$ as in \cref{Sect/framedQuivRep}. We say that $M$ is stable if, for any collection of submodules $N_i\subset M_i$, $i\in Q_0$ such that:
\begin{itemize}
\item for all $i\in Q_0$, $N_i\subset \Ker(t_i)$,
\item in restriction to $(\overline{Q},\bfm)$, the modules $N_i$ for $ i\in Q_0$ form a subrepresentation of $M$,
\end{itemize}
we have $N_i=0$ for all $i\in Q_0$. 

For any subscheme $T \subset R(\overline{Q_\bff},\widehat{\bfm};\widehat{\bfr})$, we let $T^{\st}$ denote the set of points corresponding to stable representations. In particular, we have $\Lambda_{Q_\bff,\widehat{\bfm};\widehat{\bfr}}^\st$ and $Z_{\mathrm{filt}}^\st$, where $Z_\mathrm{filt}$ is as in \cref{def/typeCrystalComponent}.
\end{definition}

As one would expect, the stable locus is an open subset of $R(\overline{Q_\bff},\widehat{\bfm};\widehat{\bfr})$.

\begin{lemma}\label{Lem/opennessStability}
The subset $R(\overline{Q_\bff},\widehat{\bfm};\widehat{\bfr})^{\st}\subset R(\overline{Q_\bff},\widehat{\bfm};\widehat{\bfr})$ is open and $\GL_{\bfm,\bfr}$-invariant.
\end{lemma}

\begin{proof}
We claim that $R(\overline{Q_\bff},\widehat{\bfm};\widehat{\bfr})^{\st}$ is the open complement of the (closed) image of the morphism
\[
Z\subset
\bigsqcup_{\bfd\ne\mathbf{0}}\mathrm{Gr}_{\bfd,\bfm}(M)\times R(\overline{Q_\bff},\widehat{\bfm};\widehat{\bfr})
\rightarrow
R(\overline{Q_\bff},\widehat{\bfm};\widehat{\bfr}),
\]
where $\mathrm{Gr}_{\bfd,\bfm}(M)$ is the Grassmannian of (not necessarily locally free) $k_{m_i}$-submodules $N_i\subset M_i$ of dimension $d_i$ over $k$ (for $i\in Q_0$) and $Z$ is the (say, reduced) closed subscheme of pairs $(N,x)$ such that $x_a(N_i)\subset N_j$ and $y_a(N_j)\subset N_i$ for all arrows $a \colon i\rightarrow j$ in $Q_1$ and $N_i\subset\Ker(t_i)$. The disjoint union runs over the dimension vectors $\bfd\ne\mathbf{0}$ such that $d_i\leq m_ir_i$ for all $i\in Q_0$.

Since $\mathrm{Gr}_{\bfd,\bfm}(M)$ is proper, the image of this morphism is closed. The closed points of its open complement correspond exactly to stable, rank $\widehat{\bfr}$ representations of $(\overline{Q_\bff},\widehat{\bfm})$, by construction.
\end{proof}

The stable locus also has trivial stabilisers, hence we deduce:

\begin{proposition}\label{Prop/trivialStabilisers}
The quotient stack $\left[R(\overline{Q_\bff},\widehat{\bfm};\widehat{\bfr})^\st/\GL_{\bfm,\bfr}\right]$ is an algebraic space.
\end{proposition}

\begin{proof}
By \cite[Thm.\ 2.2.5]{Con06}, it is sufficient to show that the stabiliser group schemes of geometric points of $\left[R(\overline{Q_\bff},\widehat{\bfm};\widehat{\bfr})^\st/\GL_{\bfm,\bfr}\right]$ are all trivial.

Let $\kappa$ be an algebraically closed field containing $k$, let $(x,y,s,t)\in R(\overline{Q_\bff},\widehat{\bfm};\widehat{\bfr})^\st(\kappa)$ and $M=M(x,y,s,t)$. Since $\Stab_{\GL_{\bfm,\bfr}(\kappa)}(x,y,s,t)$ is a Zariski-open subset of the affine space $\mathrm{Id}+\Stab_{\mathfrak{gl}_{\bfm,\bfr}}(x,y,s,t)$, it is connected. Then by \cite[Prop.\ 10.15, Prop.\ 10.31]{Mil17}, it is enough to show that $\Stab_{\mathfrak{gl}_{\bfm,\bfr}}(x,y,s,t)=0$. Consider $\xi\in\mathfrak{gl}_{\bfm,\bfr}$ such that $\xi\cdot (x,y,s,t)=0$. Then the images of $\xi_i$ for each $ i\in Q_0$ form a subrepresentation of $M$ on restriction to $(\overline{Q},\bfm)$, as $\xi$ commutes with $\epsilon\in k_{m_i}$ for all $i\in Q_0$ and with $x_a,y_a$ for all $a\in Q_1$. Moreover, for every $i\in Q_0$, we have $t_i\circ \xi_i=0$, so $t_i$ vanishes on the image of $\xi_i$. Since $M$ is stable, we obtain $\xi=0$ as required.
\end{proof}

We now collect a few lemmas on the behaviour of stability under Hecke modifications. This is crucial for our inductive quotient procedure (see \cref{Sect/RecursiveQuotient}).

\begin{lemma}\label{Lem/stabilityCriterion}
Let $M$ be a locally free, rank $\widehat{\bfr}$ representation of $(\overline{Q_\bff},\widehat{\bfm})$. If $M$ is stable, then for all $i\in Q_0$, the homomorphism of $k_{m_i}$-modules
\[
M_{i,\mathrm{out}} \colon M_i\rightarrow\bigoplus_{\substack{a\in (\overline{Q_\bff})_1 \\ a \colon i\rightarrow j}}k_a\otimes_{k_{m_j}}M_j
\]
is injective.
\end{lemma}

\begin{proof}
The claim follows from the fact that $\sub_i(M)=\Ker(M_{i,\mathrm{out}})$ is a subrepresentation of $M$ concentrated at vertex $i$ (see \cref{Lem/propertiesK_iC_i}). In particular it is a subrepresentation of $M$ under restriction to $(\overline{Q},\bfm)$, and is contained in $\Ker(t_i)$. Therefore, stability implies that $\Ker(M_{i,\mathrm{out}})=0$. \end{proof}

\begin{lemma}\label{Lem/stabilityHeckeModification}
Let $M$ be a locally free, rank $\widehat{\bfr}$ representation of $(\overline{Q_\bff},\widehat{\bfm})$. Suppose that $M'\subset M$ is a locally free subrepresentation such that $M/M'\simeq E_i^{\oplus r}$ for some $i\in Q_0$ and $r\geq1$. Then the following are equivalent:
\begin{enumerate}[label=\emph{\roman*})]
\item $M$ is stable,
\item $M'$ is stable and $M_{i,\mathrm{out}}$ is an injective homomorphism of $k_{m_i}$-modules.
\end{enumerate}
\end{lemma}

\begin{proof}
Let us first prove i) $\Rightarrow$ ii). Suppose that $M$ is stable. Then $M'$ is also stable, since any collection of submodules $N_j\subset M'_j$ satisfying $N_j\subset\Ker(t_j)$ for all $j\in Q_0$ and forming a subrepresentation of $M'$ under restriction to $(\overline{Q},\bfm)$ has the same property with respect to $M$. Moreover, stability implies $\Ker(M_{i,\mathrm{out}})=0$ by \cref{Lem/stabilityCriterion}.

Conversely, suppose that $M'$ is stable and $M_{i,\mathrm{out}}$ is injective. Consider a collection of submodules $N_j\subset M_j$ satisfying $N_j\subset\Ker(t_j)$ for all $j\in Q_0$ and forming a subrepresentation of $M$ in restriction to $(\overline{Q},\bfm)$. Then, since $M'$ is stable, we have $N_j\cap M'_j=0$ for all $j\in Q_0$. Therefore, we have $N_i\subset\Ker(M_{i,\mathrm{out}})$, hence $N_i=0$. This shows that $M$ is stable.
\end{proof}

A converse to \cref{Lem/stabilityCriterion} holds for E-filtered representations.

\begin{lemma}\label{Lem/nilpotentStabilityCondition}
Let $M$ be a locally free, rank $\widehat{\bfr}$ representation of $(\overline{Q_\bff},\widehat{\bfm})$. Suppose that $M$ is E-filtered. If $M_{i,\mathrm{out}}$ is injective for all $i\in Q_0$, then $M$ is stable.
\end{lemma}

\begin{proof}
We prove the claim by induction on $\rk M$. Since $M$ is E-filtered, there exists a (locally free) subrepresentation $M'\subset M$ such that $M/M'\simeq E_i^{\oplus r}$ for some integer $r\geq1$. If $i\in Q_0$, we may assume that $M'$ is stable, by induction, and the conclusion follows from \cref{Lem/stabilityHeckeModification}. On the other hand, if $i=\infty$, then $t_i(M_i)=0$ for every $i\in Q_0$ and the modules $M_i,\ i\in Q_0$ form an E-filtered representation $N$ of $(\overline{Q},\bfm)$ such that $N_{i,\mathrm{out}}$ is injective for every $i\in Q_0$. In that case, we must have $N=0$, as we would otherwise obtain a subrepresentation $E_i\subset N$ for some $i\in Q_0$, which contradicts the injectivity of $M_{i,\mathrm{out}}$. This implies stability in that case as well, which finishes the proof.
\end{proof}

Finally, we show that for stable E-filtered representations of $(\overline{Q_\bff},\widehat{\bfm})$, the framing maps starting from the vertex $\infty$ vanish. This is crucial to relate the geometric constructions of $B(C;\infty)$ and $B(C;\lambda)$ (see \cref{Sect/selectCrystalComponents}). We will also use this for the base case of our quotient construction (see \cref{prop/quotientBaseStep}).

\begin{lemma}\label{Lem/framingVanishing}
Let $M$ be a stable, E-filtered representation of $(\overline{Q_\bff},\widehat{\bfm})$ of rank $\widehat{\bfr}$. Then $M$ admits a subrepresentation isomorphic to $E_\infty$. In particular, $s_i=0$ for all $i\in Q_0$.
\end{lemma}

\begin{proof}
By assumption, $M$ is E-filtered, so there exist locally free subrepresentations $M''\subset M'\subset M$ such that $M'/M''\simeq E_\infty$. Then, in restriction to $(\overline{Q},\bfm)$, the modules $M''_i$ for $i\in Q_0$ form a subrepresentation of $M$ and we further have $M''_i\subset\Ker(t_i)$. Therefore, stability implies that $M''=0$, hence $M'\simeq E_\infty$ is a subrepresentation.
\end{proof}

\subsection{A toy model for the quotienting procedure}\label{Sect/toyModel}

For natural numbers $n$ and $N$ with $n <N$,  consider the action of the reductive group
\[ \GL_n \curvearrowright X_n:=\mathrm{Hom}_k(k^{\oplus n},k^{\oplus N})\]
by $g \cdot x = x g^{-1}$. For the character $\theta := \det^{-1} \colon \GL_n \rightarrow \GG_m$, the twisted affine GIT quotient is a Zariski-locally trivial quotient
\[ X_n^{\theta-(s)s}  = X_n^{\mathrm{inj}}:=\{ x \in X_n : \rk(x) = n \} \rightarrow  X_n^{\mathrm{inj}}/\GL_n \cong \GL_N /P_{n,N}\cong \mathrm{Gr}({n,N}) \]
where $\mathrm{Gr}({n,N})$ is the Grassmannian of $n$-dimensional subspaces in $k^{\oplus N}$ and $P_{n,N} < \GL_N$ is the parabolic subgroup of automorphisms preserving the inclusion $k^{\oplus n} \subset k^{\oplus N}$ as the first $n$ coordinates. The above isomorphisms holds as restriction to $k^{\oplus n}$ yields an equivariant surjection $\GL_N\twoheadrightarrow X_n^{\mathrm{inj}}$ with respect to the quotient morphism $P_{n,N}\twoheadrightarrow\GL_n$, which induces the required isomorphism $\GL_N/P_{n,N}\overset{\sim}{\longrightarrow}X_n^{\mathrm{inj}}/\GL_n$.

We will now restrict this action to a parabolic subgroup of $\GL_n$ and construct a quotient by this non-reductive action. For this, we fix decompositions $n = n' + n''$ and let $P := P_{n',n}$ be the parabolic subgroup preserving $k^{\oplus n'} \subset k^{\oplus n}$. More precisely, this is the parabolic group associated to the one-parameter subgroup
\[ \lambda_{n'}(t) = \left( \begin{array}{cc} t^0 I_{n'} & 0 \\ 0 & t^{-1}I_{n''} \end{array} \right)\]
and $P$ corresponds to block upper triangular matrices. We have $P = U \rtimes (G' \times G'')$, where $U$ is its unipotent radical and the Levi factors are given by $G' = \GL_{n'}$ and $G''= \GL_{n''}$. We will write points in $X_n = X_{n'} \times X_{n''}$ as pairs $x = (x',x'')$, and let $f$ denote the projection onto the first factor. Then there is an equivariant action
\begin{equation}\label{eqn:toy example equivariant action set up}
(P \twoheadrightarrow G' \times G'') \curvearrowright (X_{n} \stackrel{f}{\rightarrow} X_{n'})
\end{equation}
where $G''$ acts trivially on $X_{n'}$. We will construct a quotient of the $P$-action on $X$ by explicitly finding slices for the unipotent group action, and we will show this coincides with the relative non-reductive GIT quotient in \cite{HHJ}.

\begin{lemma}\label{lemma toy model without mult}
For the action of the parabolic subgroup $P$ on $X_n$, the open subset $X_n^{\mathrm{inj}}$ admits a Zariski-locally trivial $P$-quotient
\[ X_n^{\mathrm{inj}} \rightarrow \mathrm{Fl}(n',n,N)\]
where $\mathrm{Fl}(n',n,N)$ is the flag variety parametrising flags in $k^{\oplus N}$ of dimension $n' < n < N$.
\end{lemma}
\begin{proof}
We simplify notation and write $X:=X_n \stackrel{f}{\rightarrow} Z:=X_{n'}$. First we perform a base change over the stable locus $Z':= Z^{\theta-s}(G') = X_{n'}^{\mathrm{inj}}$ and denote this by $f' \colon X' \rightarrow Z'$. We claim there is a Zariski-locally trivial $U$-quotient $q_U \colon X' \rightarrow Y := X'/U$ that is relative to $Z'$ by producing slices locally. The stable locus $Z'$ is covered by open subschemes 
\[ Z_S = \{ z \in Z: S \oplus \mathrm{Im}(z) = k^{\oplus N} \}, \]
where $\dim S = N -n'$. Then we claim that for the $U$-action on $X_S:=X \times_Z Z_S$, there is a $U$-slice in $X_S$ relative to $Z_S$ given by
\[ Y_S := Z_S \times F_S, \quad \text{where} \quad F_S:=\mathrm{Hom}(k^{\oplus n''}, S).\]
Indeed, the action of $u \in U$ on $x = (x' | x'') \in X_S$ is given by $ u \cdot x = (x'| x'' - x'u)$, 
and so there is a unique element $u \in U$ such that the image of $(u \cdot x)''$ lies in the complement $S$ to the image of $x'$. In particular, we obtain trivial $U$-quotients relative to $Z_S$
\[ X_S \rightarrow X_S / U  \cong Y_S \rightarrow Z_S\]
which glue to give the claimed quotient $q_U$. 

The open locus $Y^{\mathrm{inj}}_S:= Z_S \times F_S^{\mathrm{inj}} \subset Y_S$ admits a Zariski-locally trivial $G''$-quotient relative to $Z_S$, with fibre $\mathrm{Gr}(n'',S)$. The group $G'$ acts freely and equivariantly on the morphism
\[ Y^{\mathrm{inj}}_S/G'' = Z_S \times \mathrm{Gr}(n'',S) \rightarrow Z_S\]
with Zariski-locally trivial quotient 
\[ Y^{\mathrm{inj}}_S/(G' \times G'') \cong Z_S/G' \times \mathrm{Gr}(n'',S) \hookrightarrow \mathrm{Fl}(n',n,N)\]
relative to the open subset $Z_S/G' \subset \mathrm{Gr}(n',N)$. We claim that 
\begin{equation}\label{eqn:HM type description for toy ex}
    X_n'':=\bigcup_S q_U^{-1}(Y_S^{\mathrm{inj}}) = X_n^{\mathrm{inj}}
\end{equation}
and thus $X_n^{\mathrm{inj}}$ has a Zariski-locally trivial $P$-quotient given by 
\[ X_n^{\mathrm{inj}} / P \cong \GL_N / P_{n',n,N} \cong \mathrm{Fl}(n',n,N),\]
where the first isomorphism follows as in the case of the Grassmannian above. 
To prove \cref{eqn:HM type description for toy ex}, we need to show for $x = (x',x'') \in X'$ that $x$ is injective if and only if $(u \cdot x)''$ is injective for all $u \in U$. We note that there exists $u \in U$ such that $(u \cdot x)'' = (x'' - x' u)$ is not injective if, and only if, there exists $w \in k^{\oplus n''}\setminus\{0\}$ such that $(x'' - x' u)(w) = 0$, which is in turn equivalent to the existence of $v=uw \in k^{\oplus n'}$ and $w \in k^{\oplus n''}\setminus\{0\}$ such that $x'(v) = x''(w)$. Since we assumed $x'$ to be injective, the latter statement is equivalent to $x=(x',x'')$ not being injective, so we are done.
\end{proof}

\begin{remark}
    The above explicitly constructed quotient coincides with a relative non-reductive quotient associated to the equivariant action \cref{eqn:toy example equivariant action set up} on the affine morphism $f \colon X_n \rightarrow X_{n'}$ in the sense of \cite{HHJ}. Indeed this action is (internally) graded by $\lambda_{n'}$ and the relative quotient is performed with respect to the reductive GIT quotient $q \colon X'_{n'} :=X_{n'}^{G'-ss}(\theta) = X_{n'}^{\mathrm{inj}} \rightarrow X_{n'} \git_\theta G' = \mathrm{Gr}(n',N)$ and the character $\rho \colon P \rightarrow \GG_m$ given by $\rho(g',g'') := \det(g'')^{-1}$. Using the non-reductive Hilbert--Mumford criterion, one can show that 
    \[ X_n^{(s)s}(f;q,\rho) = X_n^{\mathrm{inj}}\]
    admits a Zariski-locally trivial $P$-quotient that is projective over  $X_{n'} \git_\theta G' = \mathrm{Gr}(n',N)$. 
\end{remark}

We now consider a jet version of the above set up and look at the action 
\begin{equation}\label{eq NRGIT set-up for jet Grassmannian toy model}
    \GL_{m,n} \curvearrowright X_{m,n}:=\mathrm{Hom}_{k_m}(k_m^{\oplus n},k_m^{\oplus N}).
\end{equation}
In \cite[$\S$4.6.1]{HJV25}, we use relative non-reductive GIT to construct a Zariski-locally trivial quotient of the open set $X_{m,n}^{\mathrm{inj}}$ which is equal to the jet Grassmannian $\mathrm{Gr}_m(n,N)$ parametrising free rank $n$ submodules of $k_m^{\oplus N}$. In fact, this is the bundle of $(m-1)$-jets over $\mathrm{Gr}(n,N)$ and as $\GL_{m,n}$ is special, this quotient is Zariski-locally trivial.

As above, we write $n = n' + n''$ and restrict our attention to the action of the jet parabolic subgroup $P_m:=P_{m;n',n}$. We can write $P_m = U_m \rtimes (G_m' \times G_m'')$, where $G_m' := \GL_{m,n'}$ and $G_m'':=\GL_{m,n''}$ are non-reductive groups, and we view their product as a mock-Levi subgroup of $P_m$, as it is the centraliser subgroup of the one-parameter subgroup $\lambda_{n'} \colon \GG_m \rightarrow \GL_n\hookrightarrow \GL_{m,n}$. 

\begin{lemma}\label{Lem/toyJetParabolicQuotient}
The following statements hold for the action of $P_m$ on $X_{m,n}$.
\begin{enumerate}[label=\emph{\roman*})]
\item The open subset $X_{m,n}^{\mathrm{inj}}$ admits a Zariski-locally trivial $(U_m \rtimes G_m'')$-quotient
\[ X_{m,n}^{\mathrm{inj}} \rightarrow X_{m,n}^{\mathrm{inj}}/ (U_m \rtimes G_m''), \]
which is $G_m'$-equivariant relative to $X_{m,n'}^{\mathrm{inj}}$, with fibres given by jet Grassmannians $\mathrm{Gr}_m(n'',N-n')$.
\item There is a Zariski-locally trivial $P_m$-quotient  $X_{m,n}^{\mathrm{inj}} \rightarrow \mathrm{Fl}_m(n',n,N)$  given by the jet flag variety parametrising flags in $k_m^{\oplus N}$ of free submodules of ranks $n' < n < N$.
\end{enumerate}
\end{lemma}
\begin{proof}
As in \cref{lemma toy model without mult} above, we consider the projection $f \colon X:=X_{m,n} \rightarrow Z:=X_{m,n'}$. The open set $X_{m,n'}^{\mathrm{inj}}$ is covered by $Z_S := \{ z \in Z \colon S \oplus \mathrm{Im}(z) = k_m^{\oplus N}\}$, 
where $S$ is a free submodule of rank $N-n'$. The same argument shows that there is a $U_m$-slice on $X_S=f^{-1}(Z_S)$ given by $Y_S := Z_S \times F_S$, where $F_S:= \Hom_{k_m}(k_m^{\oplus n''},S)$. Then $Y_S^{\mathrm{inj}} = Z_S \times F_S^{\mathrm{inj}}$ admits a Zariski-locally trivial $G''_m$-quotient, which is relative and $G'_m$-equivariant over $Z_S$ with fibres $\mathrm{Gr}_m(n'',S)$. The first claim then follows by taking the preimage of the loci $Y_S^{\mathrm{inj}}$ under the $U$-quotient, and the second claim follows by additionally quotienting by $G'_m$ exactly as in \cref{lemma toy model without mult}.
\end{proof}

\subsection{A recursive quotienting procedure}\label{Sect/RecursiveQuotient}

From now on, we fix a quiver with multiplicities $(Q_\bff,\widehat{\bfm})$, where $\bff\in\bbZ_{\geq0}^{Q_0}$ is a framing vector. We will construct a quotient of the stable locus in $\Lambda_{\mathrm{filt}}$ by using the covering in \cref{eqn stratification of Lambda filt} by the type $F$, which encodes how the $E$-filtered representation is constructed as a sequence of Hecke modifications (see \cref{def/filtrationType}). Our strategy is to prove this by induction on the length of the type $F$. The main result is as follows.

\begin{theorem}\label{Thm/quotientConstruction}
For a type $F=(c_1\bfalpha_{i_1},\ldots,c_s\bfalpha_{i_s})$ with rank vector $\bfr:=\sum_{t=1}^sc_t\bfalpha_{i_t}$, the $\GL_{\bfm,\bfr}$-action on $\Lambda_F^\st$ admits a Zariski-locally trivial quotient such that the quotient map is a Zariski-locally trivial principal $\GL_{\bfm,\bfr}$-bundle. As a consequence, the same holds for the action $\GL_{\bfm,\bfr}\curvearrowright Z_{\mathrm{filt}}^\st$, for any closed crystal component $Z\subset\Lambda_{\widehat{\bfr}}$.
\end{theorem}

\begin{proof}
This result is proved by induction on the length $s$ of $F$, where the base case is \cref{prop/quotientBaseStep} and the inductive step is \cref{prop/quotientInductiveStep} below.
\end{proof}

Both the base case and the inductive step use the above toy model. Note that, by \cref{Lem/framingVanishing}, $\Lambda_F^\st\ne\emptyset$ only if the last entry of $F$ is $\bfalpha_{\infty}$. Therefore, we take as our base case $F=(c\bfalpha_i,\bfalpha_\infty)$ for some $i\in Q_0$ and $c\in\bbZ_{\geq1}$.

\begin{proposition}[Base case]\label{prop/quotientBaseStep}
Let $c\in\bbZ_{\geq1}$, $i\in Q_0$ and $\bfr:=c\bfalpha_i$. Then the action $\GL_{m_i,c}\curvearrowright\Lambda_{\widehat{\bfr}}^\st$ admits a Zariski-locally trivial quotient.
\end{proposition}
\begin{proof}
A stable E-filtered rank $\widehat{\bfr}$ representation of $\overline{Q_{\mathbf{f}}}$ satisfying the moment map equations is given by an injective $k_{m_i}$-linear map $k_{m_i}^{\oplus c} \rightarrow k_{m_i}^{\oplus f_i}$, as by \cref{Lem/framingVanishing} the maps from the framing vertex $\infty$ to $i$ are all zero. Hence we are considering the action of $\GL_{m_i,c}$ on $\Lambda_{\widehat{\bfr}}^\st = \Hom_{k_{m_i}}^{\mathrm{inj}}(k_{m_i}^{\oplus c} , k_{m_i}^{\oplus f_i})$ as in \cref{eq NRGIT set-up for jet Grassmannian toy model}. By \cite[$\S$4.6.1]{HJV25}, 
there is a Zariski-locally trivial $\GL_{m_i,c}$-quotient of $\Lambda^{\st}_{\widehat{\bfr}} = \Hom_{k_{m_i}}^{\mathrm{inj}}(k_{m_i}^{\oplus c} , k_{m_i}^{\oplus f_i})$ given by the jet Grassmannian $\mathrm{Gr}_{m_i}(c,f_i)$ parametrising free rank $c$ submodules of $k_{m_i}^{\oplus f_i}$. 
\end{proof}

\begin{notation}\label{not/type and subtype}
For the rest of this section, we fix a type $F$ of length $s$ and write $F = (c\bfalpha_i,F')$ where $F'$ is a type of length $s-1$. We let $\bfr$ and $\bfr'$ denote the rank vectors associated to the types $F$ and $F'$, so that $\bfr = \bfr' + c\bfalpha_i$.
\end{notation}

\begin{definition}\label{def/jetParabolicAction}
For a free $k_{m_i}$-submodule $N_i\subset k_{m_i}^{\oplus r_i}$ of rank $r_i-c$ and an isomorphism $N_i\simeq k_{m_i}^{\oplus (r_i-c)}$, we consider the affine subspace:
\[ R(\overline{Q_\bff},\widehat{\bfm};\widehat{\bfr},N_i) 
:= \left\{
(x,y,s,t)\in R(\overline{Q_\bff},\widehat{\bfm};\widehat{\bfr})
\ \vert\ 
\Ima(M(x,y,s,t)_{i,\mathrm{in}})\subset N_i
\right\}  \subset  R(\overline{Q_\bff},\widehat{\bfm};\widehat{\bfr}).\]
Then we define $Y_F=Y_F(N_i):=\Lambda_F\cap R(\overline{Q_\bff},\widehat{\bfm};\widehat{\bfr},N_i)$ and
let $P_{\bfm;\bfr,c\bfalpha_i}=P_{\bfm,\bfr}(N_i)\subset\GL_{\bfm,\bfr}$ denote the subgroup of $\GL_{\bfm,\bfr}$ stabilising the flag $N_i\subset k_{m_i}^{\oplus r_i}$. 
\end{definition}

Note that both $Y_F$ and $P_{\bfm;\bfr,c\bfalpha_i}$ depend on $N_i$, but as we will fix $N_i$ below, we suppress this in the notation. The action $\GL_{\bfm,\bfr}\curvearrowright\Lambda_F$ induces an action $P_{\bfm;\bfr,c\bfalpha_i}\curvearrowright Y_F$. The first step in the inductive procedure replaces the former action with the latter, by relating their quotient stacks. First, we verify that the group $P_{\bfm;\bfr,c\bfalpha_i}$ is special.

\begin{lemma}\label{Lem/specialGroups}
Let $\bfr\in\bbZ_{\geq0}^{Q_0}$ and $(c,i)\in\bbZ_{\geq0}\times Q_0$. The group $P_{\bfm;\bfr,c\bfalpha_i}$ is special.
\end{lemma}

\begin{proof}
The group $P_{\bfm;\bfr,c\bfalpha_i}$ is an extension of $\GL_{\bfm,\bfr'}\times\GL_{m_i,c}$ by a unipotent algebraic group. Thus by \cite[Lem.\ 6, Prop.\ 14]{Ser58}, it is enough to show that $\GL_{\bfm,\bfr'}\times\GL_{m_i,c}$ is special. This in turn follows from \cite[Lem.\ 6, Prop.\ 14]{Ser58} and the fact that the general linear group is special \cite[Lem.\ III.4.10]{Mil80}.
\end{proof}

We can relate the two actions by showing the quotient stacks are isomorphic.

\begin{lemma}\label{Lem/parabolicReduction}
There is a $\GL_{\bfm,\bfr}$-equivariant isomorphism of schemes:
\[
\Lambda_F\simeq \GL_{\bfm,\bfr}\times^{P_{\bfm;\bfr,c\bfalpha_i}}Y_F.
\]
In particular, there is an isomorphism of stacks $[\Lambda_F/\GL_{\bfm,\bfr}] \cong [ Y_F /P_{\bfm;\bfr,c\bfalpha_i}]$.
\end{lemma}

\begin{proof}
Let us write $Y:=Y_F$, $X:=\Lambda_F$, $P:=P_{\bfm;\bfr,c\bfalpha_i}$ and $G:=\GL_{\bfm,\bfr}$ for short. The action $G\curvearrowright X$ induces a $P$-invariant map $G\times Y\rightarrow X$ (for the diagonal action $p\cdot (g,y)=(gp^{-1},p\cdot y)$), hence a map $G\times^PY\rightarrow X$. Note that the quotient $G\times^PY$ is a scheme, as $P$ is special by \cref{Lem/specialGroups}. 

To construct an inverse morphism, let $V\subset G/P$ be a trivialising open for the principal $P$-bundle $G\rightarrow G/P$, let $s\colon V\hookrightarrow G$ be the associated section and $U\subset X$ be the inverse image of $V$ under the morphism
\[
\begin{array}{rcl}
 X & \stackrel{h}{\longrightarrow} & G/P \\
(x,y,s,t) & \mapsto & \Ima(M(x,y,s,t)_{i,\mathrm{in}})
.
\end{array}
\]
Then we obtain a morphism:
\[
\begin{array}{rcl}
U & \rightarrow & G\times Y \\
p & \mapsto & \left((s\circ h)(p),(s\circ h)^{-1}(p)\cdot p\right)
.
\end{array}
\]
After composing with the quotient map $G\times Y\rightarrow G\times^PY$, the above maps (for covering trivialising opens $V\subset G/P$) glue to a morphism $X\rightarrow G\times^PY$, which is the required inverse.
\end{proof}

\begin{remark}\label{Rmk/HeckeCorrespondences}
As we hinted in the Introduction, the map $p\colon \Lambda_{\widehat{\bfr'},\widehat{\bfr};c\bfalpha_i}\rightarrow\Lambda_{\widehat{\bfr},c\bfalpha_i}$ restricts to a principal $\GL_{m_i,r'_i}$-bundle above $\Lambda_F$, whose fibres parametrise isomorphisms:
\[
\varphi \colon k_{m_i}^{\oplus r'_i}\simeq \Ima(M_{i,\mathrm{in}})
.
\]
Therefore, we obtain a Hecke correspondence:
\[
[\Lambda_{F'}/\GL_{\bfm,\bfr'}]\overset{q}{\longleftarrow} \Lambda_{F',F}:=p^{-1}(\Lambda_F)/\GL_{m_i,r'_i}\overset{\sim}{\longrightarrow}\Lambda_F
,
\]
where $q$ is $\GL_{\bfm,\bfr}$-invariant. The projection $\Lambda_{F',F}\rightarrow\GL_{\bfm,\bfr}/P_{\bfm;\bfr,c\bfalpha_i}, M\mapsto \Ima(\varphi)$ is a fibre bundle with fibre isomorphic to $Y_F$. Reducing to the quotient stack $[Y_F/P_{\bfm;\bfr,c\bfalpha_i}]$ amounts to fixing $\Ima(M_{i,\mathrm{in}})=N_i$, as shown in the above lemma. This is summarised in the following diagram, where we quotient by $\GL_{\bfm,\bfr}$:
\[
\begin{tikzcd}
\left[Y_F/P_{\bfm;\bfr,c\bfalpha_i}\right] \ar[d]\ar[r,"\sim"] & \left[\Lambda_{F',F}/\GL_{\bfm,\bfr}\right] \ar[d] \\
\mathrm{B}P_{\bfm;\bfr,c\bfalpha_i} \ar[r,"\sim"] & \left[\GL_{\bfm,\bfr}\setminus\GL_{\bfm,\bfr}/P_{\bfm;\bfr,c\bfalpha_i}\right]
\end{tikzcd}
\]
This kind of reduction is common in the study of Hecke correspondences and Hall algebras, see for instance \cite{SV13a,Hen24a} or more recently \cite{GGS26}.
\end{remark}

\begin{proposition}[Inductive step]\label{prop/quotientInductiveStep}
Let $F = (c\bfalpha_i,F')$ be a type with rank vector $\bfr$ and $\bfr':=\bfr - c \bfalpha_i$ be the rank vector of the type $F'$.  
Suppose that the action $\GL_{\bfm,\bfr'}\curvearrowright\Lambda_{F'}^\st$ admits a Zariski-locally trivial $\GL_{\bfm,\bfr'}$-quotient. Then the action $P_{\bfm;\bfr,c\bfalpha_i}\curvearrowright Y_F^\st$ admits a Zariski-locally trivial $P_{\bfm;\bfr,c\bfalpha_i}$-quotient. Consequently, there is a Zariski-locally trivial $\GL_{\bfm,\bfr}$-quotient of $\Lambda^\st_F$.
\end{proposition}

The following subsection is devoted to the proof of this proposition.

\subsection{Proof of the inductive step}

We continue to write the type as $F = (c\bfalpha_i,F')$ as in \cref{not/type and subtype}, and let $\bfr$ and $\bfr'$ denote the rank vectors of $F$ and $F'$. Using the fixed isomorphism $N_i\simeq k_{m_i}^{\oplus (r_i-c)}$ (see \cref{def/jetParabolicAction}), we can define a restriction map:
\[
\begin{array}{rcl}
R(\overline{Q_\bff},\widehat{\bfm};\widehat{\bfr},N_i) &
\overset{\Phi}{\longrightarrow} &
R(\overline{Q_\bff},\widehat{\bfm};\widehat{\bfr}') \\
(x,y,s,t) & \mapsto & (x\vert_N,y\vert_N,s\vert_N,t\vert_N)
\end{array}
,
\]
where $N$ denotes the collection of submodules $N_j\subset k_{m_j}^{\oplus r_j},\ j\in Q_0$ given by:
\[
N_j:=
\left\{
\begin{array}{ll}
N_i & j=i, \\
k_{m_j}^{\oplus r_j} & j\ne i.
\end{array}
\right.
\]

As in \cref{Notation/0alpha_i}, we will write $0\bfalpha_i$ to keep track of the vertex $i\in Q_0$ where $\fac_i(M)=0$. In particular, $\Lambda_{F',0\bfalpha_i}$ is the open subscheme of $\Lambda_{F'}$ consisting of representations $M$ such that $M_{i,\mathrm{in}}$ is surjective.

\begin{lemma}\label{Lem/extensionMap}
The restriction map $ \Phi \colon 
R(\overline{Q_\bff},\widehat{\bfm};\widehat{\bfr},N_i)
\longrightarrow
R(\overline{Q_\bff},\widehat{\bfm};\widehat{\bfr}')$
induces a vector bundle
\[
Y_F\rightarrow\Lambda_{F',0\bfalpha_i},
\]
which is equivariant with respect to the morphism of algebraic groups $P_{\bfm;\bfr,c\bfalpha_i}\rightarrow\GL_{\bfm,\bfr'}$.
\end{lemma}

\begin{proof}
Note that the map $\Phi$ is a trivial vector bundle with fibre:
\[
\Hom_{k_{m_i}}\left( k_{m_i}^{\oplus c},\bigoplus_{a:i\rightarrow j}k_a\otimes_{k_{m_j}}k_{m_j}^{\oplus r_j}\right)
.
\]
Consider the pullback of that vector bundle to $\Lambda_{F',0\bfalpha_i}$:
\[
\begin{tikzcd}
\Phi^{-1}(\Lambda_{F',0\bfalpha_i}) \ar[r]\ar[d]\arrow[dr, phantom, "\ulcorner", very near start] & \Lambda_{F',0\bfalpha_i} \ar[d,hookrightarrow] \\
R(\overline{Q_\bff},\widehat{\bfm};\widehat{\bfr},N_i) \ar[r,"\Phi"] & R(\overline{Q_\bff},\widehat{\bfm};\widehat{\bfr}') .
\end{tikzcd}
\]
Then it suffices to prove that $Y_F\subset\Phi^{-1}(\Lambda_{F',0\bfalpha_i})$ is a subbundle.

By construction, we have:
\[
\Phi^{-1}(\Lambda_{F',0\bfalpha_i})\cap \mathrm{nil}_{\widehat{\bfr}}
=
\Phi^{-1}(\Lambda_{F',0\bfalpha_i})\cap\mu_{Q_\bff,\widehat{\bfm};\widehat{\bfr}}^{-1}(0)
\subset
\mathrm{nil}_{\widehat{\bfr},c\bfalpha_i}
,
\]
i.e. $(x,y,s,t)\in\Phi^{-1}(\Lambda_{F',0\bfalpha_i})$ lies in $\mathrm{nil}_{\widehat{\bfr}}$ if and only if $\mu_{\overline{Q_\bff},\widehat{\bfm};\widehat{\bfr}}(x,y,s,t)=0$ and in that case, $(x,y,s,t)\in\mathrm{nil}_{\widehat{\bfr},c\bfalpha_i}$. In particular, for $(x,y,s,t)\in\Phi^{-1}(\Lambda_{F',0\bfalpha_i})\cap\mathrm{nil}_{\widehat{\bfr}}$, the representation $M(x,y,s,t)$ has type $F$. Consider now the (reduced) subscheme:
\[
\mathrm{nil}_{\widehat{\bfr},N_i;c\bfalpha_i}
:=
\mathrm{nil}_{\widehat{\bfr}',\widehat{\bfr};c\bfalpha_i}
\cap
\left(
R(\overline{Q_\bff},\widehat{\bfm};\widehat{\bfr},N_i)\times\{\varphi_i:N_i\hookrightarrow k_{m_i}^{\oplus r_i}\}
\right)
\subset
\mathrm{nil}_{\widehat{\bfr}',\widehat{\bfr};c\bfalpha_i}
.
\]
The restriction of the map $p\colon \mathrm{nil}_{\widehat{\bfr}',\widehat{\bfr};c\bfalpha_i}\rightarrow\mathrm{nil}_{\widehat{\bfr},c\bfalpha_i}$ (see \cref{def/quiverHeckeCorrespondence}) to $\mathrm{nil}_{\widehat{\bfr},N_i;c\bfalpha_i}$ coincides with the restriction of the closed immersion $R(\overline{Q_\bff},\widehat{\bfm};\widehat{\bfr},N_i)\hookrightarrow R(\overline{Q_\bff},\widehat{\bfm};\widehat{\bfr})$ to $\mathrm{nil}_{\widehat{\bfr},N_i;c\bfalpha_i}$. Likewise, the restrictions of the maps $q \colon\mathrm{nil}_{\widehat{\bfr}',\widehat{\bfr};c\bfalpha_i}\rightarrow\mathrm{nil}_{\widehat{\bfr}',0\bfalpha_i}$ and $\Phi$ to $\mathrm{nil}_{\widehat{\bfr},N_i;c\bfalpha_i}$ coincide. By \cref{Rem/restrictionHeckeCorrespondence}, we thus obtain:
\[
Y_F
=
\Phi^{-1}(\Lambda_{F',0\bfalpha_i})\cap\Lambda_{\widehat{\bfr}}
=
\Phi^{-1}(\Lambda_{F',0\bfalpha_i})\cap\mu_{\overline{Q_\bff},\widehat{\bfm};\widehat{\bfr}}^{-1}(0).
\]
Now, given $(x,y,s,t)\in\Phi^{-1}(\Lambda_{F',0\bfalpha_i})$ and the associated representation $M=M(x,y,s,t)$, we have $\mu_{\overline{Q_\bff},\widehat{\bfm};\widehat{\bfr}}(x,y,s,t)=0$ if and only if $\widetilde{M}_{i,\mathrm{in}}\circ M_{i,\mathrm{out}}=0$, by \cref{Lem/momentMapEqn}. The operator
\[
\widetilde{M}_{i,\mathrm{in}}(x\vert_N,y\vert_N,s\vert_N,t\vert_N)\circ (-)\colon \! \ \Hom_{k_{m_i}}\!\!\left(k_{m_i}^{\oplus c},\bigoplus_{a:i\rightarrow j}k_a\otimes_{k_{m_j}}k_{m_j}^{\oplus r_j} \!\right)
\rightarrow
\Hom_{k_{m_i}}(k_{m_i}^{\oplus c},k_{m_i}^{\oplus (r_i-c)})
\]
has constant rank on $\Lambda_{F',0\bfalpha_i}$, since $\widetilde{M}_{i,\mathrm{in}}(x\vert_N,y\vert_N,s\vert_N,t\vert_N)$ has constant rank, equal to $r_i-c$. This shows that $Y_F=\Phi^{-1}(\Lambda_{F',0\bfalpha_i})\cap\mu_{\overline{Q_\bff},\widehat{\bfm};\widehat{\bfr}}^{-1}(0)\subset\Phi^{-1}(\Lambda_{F',0\bfalpha_i})$ is the subbundle $\Ker(\widetilde{M}_{i,\mathrm{in}}\circ (-))$, as required.
\end{proof}

We will need the following well-known lemma in the following proofs.

\begin{lemma}\label{Lem/slicesPullback}
Let $G$ be an algebraic group and $f \colon X\rightarrow Y$ a $G$-equivariant map of $G$-schemes. Suppose that the action $G\curvearrowright Y$ admits a Zariski-locally trivial $G$-quotient. Then so does the action $G\curvearrowright X$.
\end{lemma}

\begin{proof}
This amounts to providing $G$-slices for certain covering $G$-invariant opens in $X$. We argue that these can be obtained by pulling back such slices from $Y$. Indeed, suppose that $V\subset Y$ is a $G$-invariant open subset endowed with a $G$-equivariant isomorphism $V\simeq G\times S$, $v\mapsto (g(v),s(v))$. Then the map
\[
\begin{array}{rcl}
f^{-1}(V) & \rightarrow & G\times f^{-1}(S) \\
u & \mapsto & \left((g\circ f)(u),(g\circ f)(u)^{-1}\cdot u\right)
\end{array}
\]
is an isomorphism.
\end{proof}

We first take a partial quotient by the following normal subgroup of $P_{\bfm;\bfr,c\bfalpha_i}$.

\begin{notation}
Let $U_{\bfm;\bfr,c\bfalpha_i}\subset P_{\bfm;\bfr,c\bfalpha_i}$ be the normal subgroup such that the following decomposition holds:
\[
P_{\bfm;\bfr,c\bfalpha_i}:=U_{\bfm;\bfr,c\bfalpha_i}\rtimes (\GL_{\bfm,\bfr'}\times\GL_{m_i,c})
.
\]
\end{notation}

\begin{lemma}\label{Lem/partialQuotient}
The vector bundle $Y_F\rightarrow\Lambda_{F',0\bfalpha_i}$ induces a map $Y_F^\st\rightarrow\Lambda_{F',0\bfalpha_i}^\st$, which in turn induces a fibrewise, Zariski-locally trivial quotient:
\begin{equation}\label{eq relative morphism inductive step}
    Y_F^\st/(U_{\bfm;\bfr,c\bfalpha_i}\rtimes\GL_{m_i,c})\rightarrow\Lambda_{F',0\bfalpha_i}^\st .
\end{equation}
\end{lemma}

\begin{proof}
Recall that, by the proof of \cref{Lem/extensionMap}, the map $\Phi\colon Y_F\rightarrow\Lambda_{F',0\bfalpha_i}$ may be identified with the vector bundle:
\[
\Ker(\widetilde{M}_{i,\mathrm{in}}\circ (-))=\mathcal{H}om_{k_{m_i}}(k_{m_i}^{\oplus c},\Ker(\widetilde{M}_{i,\mathrm{in}}))\subset \Lambda_{F',0\bfalpha_i}\times\Hom_{k_{m_i}}\left( k_{m_i}^{\oplus c},\bigoplus_{a:i\rightarrow j}k_a\otimes_{k_{m_j}}k_{m_j}^{\oplus r_j}\right)
.
\]

Let $V\subset\Lambda_{F',0\bfalpha_i}$ be a trivialising open subset for $\Ker(\widetilde{M}_{i,\mathrm{in}})$. Let $p:=\rk_{k_{m_i}}(\Ker(\widetilde{M}_{i,\mathrm{in}}))$. Then, by \cref{Lem/stabilityHeckeModification}, we have the following diagram:
\[
\begin{tikzcd}
\Phi^{-1}(V^\st)^\st \ar[d]\ar[r,hookrightarrow]\arrow[dr, phantom, "\ulcorner", very near start] & \Phi^{-1}(V^\st) \ar[d]\ar[r,"\Phi"]\arrow[dr, phantom, "\ulcorner", very near start] & V^\st \ar[d] \\
\Hom_{k_{m_i}}^{\mathrm{inj}}(k_{m_i}^{\oplus r_i},k_{m_i}^{\oplus p}) \ar[r, hookrightarrow] & \Psi^{-1}(\Hom_{k_{m_i}}^{\mathrm{inj}}(N_i,k_{m_i}^{\oplus p})) \ar[r,"\Psi"] & \Hom_{k_{m_i}}^{\mathrm{inj}}(N_i,k_{m_i}^{\oplus p}) ,
\end{tikzcd}
\]
where $\Psi\colon \Hom_{k_{m_i}}(k_{m_i}^{\oplus r_i},k_{m_i}^{\oplus p})\rightarrow\Hom_{k_{m_i}}(N_i,k_{m_i}^{\oplus p})$ is the restriction morphism and the vertical arrows are the $(U_{\bfm;\bfr,c\bfalpha_i}\rtimes\GL_{m_i,c})$-equivariant maps forgetting arrows of $Q$ with source different from $i$. By combining \cref{Lem/toyJetParabolicQuotient} and \cref{Lem/slicesPullback} applied to the leftmost arrow in the above diagram, we obtain a Zariski-locally trivial $(U_{\bfm;\bfr,c\bfalpha_i}\rtimes\GL_{m_i,c})$-quotient: 
\[
\Phi^{-1}(V^\st)^\st/(U_{\bfm;\bfr,c\bfalpha_i}\rtimes\GL_{m_i,c}).
\]
By gluing these quotients for different trivialising opens $V$, we obtain the required quotient $Y_F^\st/(U_{\bfm;\bfr,c\bfalpha_i}\rtimes\GL_{m_i,c})$.
\end{proof}

\begin{proof}[Proof of \cref{prop/quotientInductiveStep}] 
By assumption there is a Zariski-locally trivial $\GL_{\bfm,\bfr'}$-quotient of $\Lambda_{F'}^\st$, and thus also of the open invariant subset $\Lambda_{F',0\bfalpha_i}^\st$. Then the result about the $P_{\bfm;\bfr, c\bfalpha_i}$-quotient follows by applying \cref{Lem/slicesPullback} to the $\GL_{\bfm,\bfr'}$-equivariant morphism \eqref{eq relative morphism inductive step} constructed in \cref{Lem/partialQuotient}. The last statement follows by \cref{Lem/parabolicReduction}.
\end{proof}

This completes the proof of \cref{Thm/quotientConstruction}.

\begin{remark}\label{Rem/fibresRecursiveQuotient}
The above proof implies that the map:
\[
\Lambda_F^\st/\GL_{\bfm,\bfr}\simeq Y_F^\st/P_{\bfm;\bfr,c\bfalpha_i}\rightarrow\Lambda_{F',0\bfalpha_i}^\st/\GL_{\bfm,\bfr'}
\]
is a Zariski-locally trivial fibration with the same fibres as the map:
\[
\Hom_{k_{m_i}}^{\mathrm{inj}}(k_{m_i}^{\oplus r_i},k_{m_i}^{\oplus p})/(U_{\bfm;\bfr,c\bfalpha_i}\rtimes\GL_{m_i,c})
\rightarrow
\Hom_{k_{m_i}}^{\mathrm{inj}}(N_i,k_{m_i}^{\oplus p}),
\]
which, by \cref{Lem/toyJetParabolicQuotient}, are isomorphic to:
\[
\mathrm{Gr}_{m_i}\left( k_{m_i}^{\oplus c},\Ker(\widetilde{M}'_{i,\mathrm{in}})/\Ima(M'_{i,\mathrm{out}})\right)
=\mathrm{Gr}_{m_i}(c,p-r'_i)
=\mathrm{Gr}_{m_i}\left(c,\left\langle h_i,\lambda-\sum_{j\in Q_0}r'_j\alpha_j\right\rangle\right)
,
\]
where $M'$ is the representation corresponding to the base point in $\Lambda_{F',0\bfalpha_i}/\GL_{\bfm,\bfr'}$. Indeed, by surjectivity of $\widetilde{M}_{i,\mathrm{in}}$, we have:
\[
p=\rk\Ker(\widetilde{M}'_{i,\mathrm{in}})=f_i-\sum_{j\ne i}c_{ij}r_j-r_i+c=f_i-\sum_{j\ne i}c_{ij}r'_j-r'_i.
\]
Note that fibres are non-empty if, and only if
\[
c\leq\left\langle h_i,\lambda-\sum_{j\in Q_0}r'_j\alpha_j\right\rangle,
\]
which is consistent with \cref{Prop/crystalOperatorRange} below (note that, in the notation of that Proposition, $c=0$ and $l$ in the Proposition correspond to $c$ in this Remark).
\end{remark}

\begin{example}
Consider the quiver $(Q,\bfm)$ as in \cref{Exmp/NonMaxComponentsTypeC2} and fix $\bfr=(2,2)$, $\bff=(2,2)$. Then $\Lambda_{Q_\bff,\widehat{\bfm};\widehat{\bfr}}^\st$ parametrises representations of the form:
\[
\begin{tikzcd}
k^{\oplus 2}\ar[r,"A",shift left]\ar[d,"t_1"] & k_2^{\oplus 2} \ar[l,"B", shift left]\ar[d,"t_2"] \\
k^{\oplus 2} & k_2^{\oplus 2},
\end{tikzcd}
\]
which satisfy (but are not characterised by) the following properties: $AB=0$, $BA[\epsilon]=0$ and the linear maps
\[
\begin{array}{cccc}
M_{1,\mathrm{out}}: & k^{\oplus 2} & \overset{(A,t_1)}{\longrightarrow} & k_2^{\oplus 2}\oplus k^{\oplus 2} \\
M_{2,\mathrm{out}}: & k_2^{\oplus 2} & \overset{(B,t_2)}{\longrightarrow} & k_2^{\oplus 2}\oplus k_2^{\oplus 2}
\end{array}
\]
are injective. Note that the forgetful map $\Lambda_{Q_\bff,\widehat{\bfm};\widehat{\bfr}}^\st\rightarrow\Lambda_{Q,\bfm;\bfr}$ induces an injection on irreducible components (see \cref{Prop/geometricCriteriaB(lambda)} below for a general proof).

Consider the component $Z\subset\Lambda_{Q_\bff,\widehat{\bfm};\widehat{\bfr}}^\st$ lying above the orbit closure of
\[
(A_0,B_0)=((\begin{smallmatrix}1 & 0 \\ 0 & 0\end{smallmatrix}),(\begin{smallmatrix}0 & 0 \\ 1 & 0\end{smallmatrix})),
\]
which is a closed crystal component of $\Lambda_{Q,\bfm;\bfr}$ as discussed in \cref{Exmp/cokernelLociCrystalComponents}. Then a general representation $M(A,B,t)$ for $ (A,B,t)\in Z$ admits filtrations of types $F_1=(\bfalpha_2,\bfalpha_1,\bfalpha_2,\bfalpha_1,\bfalpha_\infty)$ and $F_2=(\bfalpha_1,2\bfalpha_2,\bfalpha_1,\bfalpha_\infty)$. One can check that $Z_{F_1}$ is the locus where $(A,B)$ lies in the orbit of $(A_0,B_0)$, whereas $Z_{F_2}$ is the locus where $\fac_1(M)$ has rank $1$. Therefore, we have $Z_{F_1}\subsetneq Z_{F_2}$. Note that, since $r_i\leq f_i$ for all $i\in Q_0$, the map
\[
Z\rightarrow\overline{\GL_{\bfm,\bfr}\cdot (A_0,B_0)}
\]
is surjective (the maps $t_i$ for $ i\in Q_0$ can be chosen to be injective).

For the type $F = F_2$, let us construct the quotient $Z_{F}/\GL_{\bfm,\bfr}$ inductively. We set $F'=(2\bfalpha_2,\bfalpha_1,\bfalpha_\infty)$, $\bfr'=(1,2)$ and $F''=(\bfalpha_1,\bfalpha_\infty)$, $\bfr''=(1,0)$. By (the proof of) \cref{prop/quotientBaseStep}, we have $Z_{F''}/\GL_{\bfm,\bfr''}\simeq\mathbb{P}^1$. Moreover, for any $(x,y,s,t)\in Z_{F''}$ and $M''=M(x,y,s,t)$ we have that $M''_{2,\mathrm{in}}$ is surjective, $\Ima(M''_{2,\mathrm{out}})=0$ and $\Ker(M''_{2,\mathrm{in}})=k_2\oplus k_2^{\oplus 2}$. Therefore, by \cref{Rem/fibresRecursiveQuotient}, the map
\[
Z_{F'}/\GL_{\bfm,\bfr'}\rightarrow Z_{F'',0\bfalpha_2}/\GL_{\bfm,\bfr''}\simeq\mathbb{P}^1
\]
is a fibration, with fibres isomorphic to $\mathrm{Gr}_2(2,3)\simeq\mathrm{T}\mathbb{P}^2$, parametrising the subspace $(B',t'_2):k_2^{\oplus 2}\hookrightarrow k_2\oplus k_2^{\oplus 2}$. Finally, for $(x,y,s,t)\in Z_{F'}$ and $M'=M(x,y,s,t)$ we have $(x,y,s,t)\in Z_{F',0\bfalpha_1}$ if and only if $B':k_2^{\oplus 2}\rightarrow k_2$ is non-zero (as $M'_{1,\mathrm{in}}(v)=(B'v)[\epsilon]$ for all $v\in k_2^{\oplus2}$). Furthermore, we have:
\begin{align*}
& \Ima(M'_{1,\mathrm{out}})=\Ima(t'_1)\subset k_2^{\oplus 2}\oplus k^{\oplus 2} \\
& \Ker(M'_{1,\mathrm{in}})=\Ker(B'(-)[\epsilon])\oplus k^{\oplus 2}\subset k_2^{\oplus 2}\oplus k^{\oplus 2}.
\end{align*}
Therefore, $Z_{F',0\bfalpha_1}\subset Z_{F'}$ is the open subset of the above $\mathrm{T}\mathbb{P}^2$-fibration which restricts on fibres to the locus where the map $k_2^{\oplus 2}\hookrightarrow k_2\oplus k_2^{\oplus 2}\twoheadrightarrow k_2$ is non-zero. Likewise, the map
\[
Z_F/\GL_{\bfm,\bfr}\rightarrow Z_{F',0\bfalpha_1}/\GL_{\bfm,\bfr'}
\]
is a $\mathbb{P}^3$-fibration, whose fibres parametrise the subspace
\[
(A,t_1):k\hookrightarrow\Ker(B'(-)[\epsilon])\oplus\mathrm{Coker}(t'_1).
\]

Summing up, $Z_F/\GL_{\bfm,\bfr}$ is a tower of two fibrations over $\mathbb{P}^1$, with respective fibres an open locus in $\mathrm{T}\mathbb{P}^2$ and $\mathbb{P}^3$.
\end{example}

\section{Prerequisites on crystals}\label{Sect/prerequisitesCrystals}

In this section, we collect prerequisite results concerning crystals and their geometric realisations \cite{KS97,Sai02,GLS18a}. In particular, we adapt Saito's strategy for constructing crystals of irreducible highest-weight representations of Kac--Moody algebras \cite{Sai02} to the setting of quivers with multiplicities. Recall that we assume $Q$ has no loop arrows.

\subsection{Kashiwara crystals}

In this section, we recall basic facts concerning Kashiwara crystals attached to a symmetrisable Cartan datum. In particular, we recall a characterisation of the crystals of highest-weight irreducible representations of Kac--Moody algebras.

Let $C = (c_{ij})_{i,j \in I}$ be a symmetrisable generalised Cartan matrix indexed by $I$, let $\mathfrak{g}$ be the associated Kac--Moody algebra and $P$ be its weight lattice (see \cite{Kac95} for background). For $i \in I$, let $\alpha_i$ (resp.\ $h_i$, $\varpi_i$) denote the corresponding simple root (resp.\ simple coroot, fundamental weight) of $\mathfrak{g}$. In particular, we have the following pairings, for all $i,j\in I$:
\begin{align}\label{pairings/rootscoroots}
\langle h_i,\alpha_j\rangle & =c_{ij}, \\
\langle h_i,\varpi_j\rangle & =\delta_{ij}. \notag
\end{align}

\begin{definition}
A $\mathfrak{g}$-crystal is the datum of a set $B$, along with maps:
\begin{itemize}
\item $\mathrm{wt} \colon B\rightarrow P$ (which we refer to as a weight function),
\item $e_i,f_i \colon B\rightarrow B\sqcup\{0\}$ for every $i\in I$ (which we refer to as raising and lowering operators),
\item $\epsilon_i,\phi_i \colon B\rightarrow\bbZ\sqcup\{-\infty\}$ for every $i\in I$,
\end{itemize}
satisfying the following axioms:
\begin{enumerate}[label=(C\arabic*)]
\item\label{C1} For all $i\in I,\ b\in B$, we have $ \phi_i(b)=\epsilon_i(b)+\langle h_i,\mathrm{wt}(b)\rangle$,
\item\label{C2} For all $b\in B$ such that $e_i(b)\in B$, we have 
$
\left\{
\begin{array}{l}
\mathrm{wt}(e_i(b))=\mathrm{wt}(b)+\alpha_i \\
\epsilon_i(e_i(b))=\epsilon_i(b)-1 \\
\phi_i(e_i(b))=\phi_i(b)+1
\end{array}
\right.
$,
\item\label{C3} For all $b\in B$ such that $f_i(b)\in B$, we have $
\left\{
\begin{array}{l}
\mathrm{wt}(f_i(b))=\mathrm{wt}(b)-\alpha_i \\
\epsilon_i(f_i(b))=\epsilon_i(b)+1 \\
\phi_i(f_i(b))=\phi_i(b)-1
\end{array}
\right.
$,
\item\label{C4} For all $ i\in I,\ b,b'\in B$ we have$\ b=e_i(b')\Leftrightarrow b'=f_i(b)$,
\item\label{C5} If $\epsilon_i(b)=-\infty$ (or equivalently, if $\phi_i(b)=-\infty$), then $e_i(b)=f_i(b)=0$.
\end{enumerate}
\end{definition}

\begin{remark}\label{Rem/crystalOperatorsE/F}
By \ref{C4}, the operators $e_i$ are determined by the operators $f_i$, and vice versa (see for instance \cite[\S7]{BK10}). Similarly, when the weight function is specified, by \ref{C1} the functions $\epsilon_i$ are determined by the functions $\phi_i$, and vice versa.
\end{remark}

\begin{notation}
Given a crystal $B$, a weight $\lambda\in P$ and $i\in Q_0$, $c\in\bbZ$, we use the following notation: $B[\lambda]:=\mathrm{wt}^{-1}(\lambda)$ and $B[\lambda]_{i,c}:=B[\lambda]\cap\epsilon_i^{-1}(c)$.
\end{notation}

The crystals we will consider in this paper all come from highest-weight representations of the quantised universal enveloping algebra of $\mathfrak{g}$. We will next introduce highest-weight crystals, and show that such a crystal is completely determined by its highest-weight elements and the lowering operators ${f}_i$, as we will use this result later in our proofs.

\begin{definition}\label{def/highestWeightCrystal}
A $\mathfrak{g}$-crystal $B$ is called highest-weight if there exists a distinguished element $b_+\in B$, called the highest-weight element, such that:
\begin{itemize}
\item for every $b\in B$, there exist $s\geq0$ and  a (possibly empty) sequence $(i_1,\ldots,i_s)\in I^s$ such that $b=({f_{i_1}}\circ\ldots\circ{f_{i_s}})(b_+)$,
\item for all $b\in B$, for all $i\in I$, we have $\epsilon_i(b)=\max\{k\geq0\ \vert\ e_i^k(b)\ne0\}$.
\end{itemize}
\end{definition}

\begin{lemma}\label{Lem/HWcrystalsF-operators}
Let $B$ be a highest-weight $\mathfrak{g}$-crystal with highest-weight element $b_+$. Then the crystal structure on $B$ is uniquely determined by the weight $\mathrm{wt}(b_+)$ of the highest-weight element, and the lowering operators $f_i$ for all $i \in I$.
\end{lemma}

\begin{proof}
The weight function is determined by its value on $b_+$, as for any $b\in B$, we can write $b={f_{i_1}}\circ\ldots {f_{i_s}}(b_+)$, then $\mathrm{wt}(b)=\mathrm{wt}(b_+)-\sum_{t=1}^s\alpha_{i_t}$. As noted in \cref{Rem/crystalOperatorsE/F}, the operators $e_i$ are determined by the operators $f_i$, and the functions $\phi_i$ are determined by $\epsilon_i$. The latter is already known, as the functions $\epsilon_i$ are determined by the operators $e_i$ as per the second part of the definition of a highest-weight crystal.
\end{proof}

The next examples introduce the two main crystals which we consider in this paper and their properties. One is attached to irreducible highest-weight representations of the quantised universal enveloping algebra of $\mathfrak{g}$, whereas the other is attached to the negative part of that quantum algebra. We will also make use of an auxiliary crystal $T(\lambda)$, which allows us to shift weights by $\lambda$ using a tensor product of crystals.

\begin{example}\label{Exmp/B(infty)}
There is a highest-weight $\mathfrak{g}$-crystal $B(C;\infty)$ (of highest weight $0$) associated to the crystal basis of the quantised enveloping algebra $\mathcal{U}_q^-(\mathfrak{g})$. We denote by $b(\infty)$ its highest-weight element. The crystal basis was constructed by Kashiwara in \cite{Kas91}. It satisfies the following property: for all $b\in B$ and $i\in I$, we have:
\[
\epsilon_i(b)=\max\{k\geq0\ \vert\ e_i^k(b)\ne0\}
.
\]
Moreover, for all $i\in I$, we have $f_i(B(C;\infty))\subset B(C;\infty)$.
\end{example}

\begin{example}\label{Exmp/B(lambda)}
Let $\lambda\in P$ be a dominant (integral) weight. There is a highest-weight $\mathfrak{g}$-crystal $B(C;\lambda)$ (of highest weight $\lambda$) associated to the crystal basis of the simple highest-weight module $L(\lambda)$ over the quantised enveloping algebra $\mathcal{U}_q(\mathfrak{g})$. We denote by $b(\lambda)$ its highest-weight element. The crystal basis was also constructed in \cite{Kas91}. It satisfies the following property: for all $b\in B$ and $i\in I$, we have:
\begin{align*}
& \epsilon_i(b)=\max\{k\geq0\ \vert\ e_i^k(b)\ne0\} , \\
& \phi_i(b)=\max\{k\geq0\ \vert\ f_i^k(b)\ne0\} .
\end{align*}
\end{example}

\begin{example}\label{Exmp/T(lambda)}
Given an (integral) weight $\lambda\in P$, there is an auxiliary crystal $T(\lambda)$ with underlying set $\{t(\lambda)\}$. The crystal operators are given by $e_i(t(\lambda))=f_i(t(\lambda))=0$, for $i\in I$ and the functions $\mathrm{wt},\epsilon_i,\phi_i$ for $ i\in I$ are given by $\mathrm{wt}(t(\lambda))=\lambda$ and $\epsilon_i(t(\lambda))=\phi_i(t(\lambda))=-\infty$.
\end{example}

There is a notion of tensor product of crystals (see \cite[\S 7.3]{Kas95}) encoding the compatibility of crystal bases with tensor products of modules over the quantised enveloping algebra. We will only use the following tensor product by $T(\lambda)$ in this paper. Note that the tensor product bifunctor on the category of crystals is not symmetric.

\begin{definition}
Let $B$ be a $\mathfrak{g}$-crystal and $\lambda\in P$ be an (integral) weight. We define $B\otimes T(\lambda)$ as the crystal with underlying set $B$, crystal operators $e_i,f_i$ for $ i\in Q_0$ being those of $B$ and functions $\mathrm{wt}',\epsilon'_i,\phi'_i$ for $ i\in Q_0$ given by:
\[
\left\{
\begin{array}{l}
\mathrm{wt'}(b)=\mathrm{wt}(b)+\lambda \\
\epsilon'_i(b)=\epsilon_i(b) \\
\phi'_i(b)=\phi_i(b)+\langle h_i,\lambda\rangle
\end{array}
\right.
,
\]
for all $b\in B$. Given $b\in B$, we denote by $b\otimes t(\lambda)$ the corresponding element in $B\otimes T(\lambda)$.
\end{definition}

Our geometric realisation of $B(C;\lambda)$ will rely on a projection between two crystals.

\begin{definition}
Let $B_1,B_2$ be $\mathfrak{g}$-crystals. A morphism of $\mathfrak{g}$-crystals from $B_1$ to $B_2$ is a map $\psi \colon B_1\sqcup\{0\}\rightarrow B_2\sqcup\{0\}$ satisfying the following conditions:
\begin{itemize}
\item $\psi(0)=0$,
\item for all $b\in B_1$, if $\psi(b)\ne0$, then $
\left\{
\begin{array}{l}
\mathrm{wt}(\psi(b))=\mathrm{wt}(b) \\
\epsilon_i(\psi(b))=\epsilon_i(b) \\
\phi_i(\psi(b))=\phi_i(b)
\end{array}
\right.
$,
\item for every $b\in B_1$ such that $\psi(b)\ne0$ and $\psi(e_i(b))\ne0$ (resp.\ $\psi(f_i(b))\ne0$), we have $e_i(\psi(b))=\psi(e_i(b))$ (resp.\ $f_i(\psi(b))=\psi(f_i(b))$).
\end{itemize}
The morphism $\psi$ is called strict if it commutes with $e_i,f_i$ for all $\ i\in I$ (with the convention that $e_i(0)=f_i(0)=0$). It is called lower semistrict if it commutes with $f_i$ for all $i\in I$.
\end{definition}

\begin{example}\label{Exmp/crystalMorphismSimpleQuotient}
The action of $\mathcal{U}_q^-(\mathfrak{g})$ on the highest-weight crystal basis vector of $L(\lambda)$ induces a morphism of $\mathfrak{g}$-crystals:
\[
\Phi:
B(C;\infty)\otimes T(\lambda)\sqcup\{0\}
\rightarrow
B(C;\lambda)\sqcup\{0\}
.
\]
In \cite[\S 2.3]{Sai02}, it is stated that this morphism of crystals is strict. It is actually only lower semistrict (see for instance \cite[\S 8]{Kas95}). Indeed, for any $i\in I$, there exists $b\in B(C;\infty)$ such that $\Phi(b)\ne0$ and $\Phi(f_i(b))=0$ (while $f_i(b)\ne0$ since $f_i(B(C;\infty))\subset B(C;\infty)$). Then $\Phi(e_i(f_i(b)))\ne0$ whereas $e_i(\Phi(f_i(b)))=e_i(0)=0$, so $\Phi$ does not always commute with $e_i$. This was pointed out to us by Jiarui Fei.
\end{example}

The following characterisation of $B(C;\lambda)$ is key to Saito's geometric construction of $B(C;\lambda)$ from the construction of $B(C;\infty)$ by Kashiwara--Saito in the symmetric case \cite{KS97,Sai02}. We will closely follow this strategy in the symmetrisable case.

\begin{proposition}\label{Prop/SaitoCriterion}
\emph{\cite[Prop.\ 2.3.1]{Sai02}}
Let $B$ be a $\mathfrak{g}$-crystal. Let $\lambda\in P$ be a dominant (integral) weight and $b_\lambda\in B$ be an element of weight $\lambda$. Suppose that:
\begin{enumerate}[label=\roman*)]
\item $b_\lambda$ is the unique element of $B$ with weight $\lambda$;
\item there is a lower semistrict surjective morphism of crystals
\[
\Phi\colon B(C;\infty)\otimes T(\lambda)\sqcup\{0\}\rightarrow B\sqcup\{0\}
\]
mapping $b(\infty)\otimes t(\lambda)$ to $b_\lambda$;
\item $\Phi$ induces a bijection between $B$ and $\{ b\otimes t(\lambda)\in B(C;\infty)\otimes T(\lambda)\ \vert\ \Phi(b\otimes t(\lambda))\ne0\}$;
\item for all $b\in B$ and $i\in I$, we have:
\begin{align*}
\epsilon_i(b) & =\max\{k\geq0\ \vert\ e_i^k(b)\ne0\} , \\
\phi_i(b) & =\max\{k\geq0\ \vert\ f_i^k(b)\ne0\} .
\end{align*}
\end{enumerate}
Then there is an isomorphism of crystals $\Psi \colon B\simeq B(C;\lambda)$ such that $\Psi\circ\Phi$ is the morphism of crystals described in \cref{Exmp/crystalMorphismSimpleQuotient}.
\end{proposition}

\begin{remark}\label{Rmk/semistrictMorphisms}
In \cite[Prop.\ 2.3.1]{Sai02}, $\Phi$ is assumed to be a strict morphism. Such a morphism $\Phi$ cannot exist, since the morphism of crystals
\[
\Phi: B(C;\infty)\otimes T(\lambda)\sqcup\{0\}\rightarrow B(C;\lambda)\sqcup\{0\}
\]
is only lower semistrict (see \cref{Exmp/crystalMorphismSimpleQuotient}). However, Saito's proof works verbatim if we assume $\Phi$ to be lower semistrict.

Indeed, if $\Phi \colon B_1\sqcup\{0\}\rightarrow B_2\sqcup\{0\}$ is a lower semistrict morphism of arbitrary crystals which induces a bijection
\[
\{b\in B_1\ \vert\ \Phi(b)\ne0\}\overset{\sim}{\longrightarrow} B_2,
\]
then one can directly check that $\Phi$ can only fail to commute with raising operators when there exists $b\in B_1$ and $i\in I$ such that $e_i(b)\ne0$, $\Phi(b)=0$ and $\Phi(e_i(b))\ne0$. Setting $B_1=B(C;\infty)\otimes T(\lambda)$ and $B_2=B$ (resp. $B_2=B(C;\lambda)$), we see that this case does not occur when checking that the bijection
\[
B\sqcup\{0\}
\overset{\sim}{\longrightarrow}
\{b\in B(C;\infty)\otimes T(\lambda)\ \vert\ \Phi(b)\ne 0\}\sqcup\{0\}
\overset{\sim}{\longrightarrow}
B(C;\lambda)\sqcup\{0\}
\]
commutes with the operators $e_i,\ i\in I$. Therefore, Saito's argument works despite the fact that $\Phi$ does not always commute with $e_i$ for $ i\in I$.
\end{remark}

\subsection{Geometric crystal operators}\label{Sect/GLScrystals}

In this section, we recall Geiss--Leclerc--Schr\"{o}er's construction of $B(C;\infty)$ for symmetrisable Cartan data using quiver representations with multiplicities \cite{GLS18a}, generalising \cite{KS97}. We also set up a geometric crystal construction for framed quivers with multiplicities, from which we will extract $B(C;\lambda)$ in \cref{Sect/selectCrystalComponents}.

Let $(Q,\bfm)$ be a quiver with multiplicities \emph{without loop arrows}. The datum of $(Q,\bfm)$, up to orientation of arrows, is equivalent to that of a symmetrisable Cartan matrix, via the following construction (see \cite{GLS17a}).

\begin{definition}
The symmetrisable Cartan matrix associated to $(Q,\bfm)$ is the square matrix $C = (c_{ij})_{i,j \in Q_0}$ with the following entries:
\[
c_{ij}=
\left\{
\begin{array}{ll}
2 & i=j \\
-\frac{m_j}{m_{ij}}\times\#\{a\in Q_1\ \vert\ \{s(a),t(a)\}=\{i,j\}\} & i\ne j
\end{array}
\right.
,
\]
for $i,j\in Q_0$. This matrix is symmetrisable with symmetriser $D=\mathrm{diag}(m_i,\ i\in Q_0)$. We let $\mathfrak{g} = \mathfrak{g}_{Q,\bfm}$ denote the associated symmetrisable Kac--Moody algebra.
\end{definition}

Note that, for $i,j\in Q_0$, we have:
\[
\rk_{k_{m_i}}\left(\bigoplus_{\{s(a),t(a)\}=\{i,j\}}k_a\right)=\frac{m_j}{m_{ij}}\times\#\{a\in Q_1\ \vert\ \{s(a),t(a)\}=\{i,j\}\}=-c_{ij}.
\]

Let us now recall Geiss--Leclerc--Schr\"{o}er's construction of the $\mathfrak{g}_{Q,\bfm}$-crystal $B(C;\infty)$. Elements of the crystal are given by closed crystal components, i.e. irreducible components of $\Lambda_{Q,\bfm}$ (see \cref{Sect/CrystalComponents}), whereas the raising and lowering operators are determined by Hecke correspondences.

\begin{definition}\label{def/GeissLeclercSchröerCrystal}
\cite[\S 3.7, \S 5.3, Thm.\ 5.4]{GLS18a}
Let $B(Q,\bfm;\infty)$ be the set $\mathrm{Irr}(\Lambda)$, endowed with the following $\mathfrak{g}_{Q,\bfm}$-crystal structure: for $\bfr\in\bbZ_{\geq0}^{Q_0}$, $i\in Q_0$ and $Z\in\mathrm{Irr}(\Lambda_\bfr)$, define:
\begin{align*}
\mathrm{wt}(Z) & := -\sum_{i\in Q_0}r_i\alpha_i, \\
\epsilon_i(Z) & := \rk \fac_i(M),
\end{align*}
for $M=M(x,y)$ and $(x,y)\in Z$ general. If $c:=\epsilon_i(Z)$, we define
\[
e_i(Z):=
\left\{
\begin{array}{ll}
0 & c=0 \\
\overline{q(p^{-1}(Z\cap\Lambda_{\bfr,c\bfalpha_i}))} & c>0,
\end{array}
\right.
\]
where
\[
\Lambda_{\bfr-\bfalpha_i,(c-1)\bfalpha_i}\overset{q}{\longleftarrow}
\Lambda_{\bfr-\bfalpha_i,\bfr;c\bfalpha_i}\overset{p}{\longrightarrow}\Lambda_{\bfr,c\bfalpha_i}
,
\]
and, by convention, the two leftmost terms are empty if $c=0$. By \cref{Rem/crystalOperatorsE/F} this data completely determines the crystal structure.
\end{definition}

\begin{notation}
Given $b\in B(Q,\bfm;\infty)$, we let $Z_b\subset\Lambda$ denote the corresponding irreducible component. Moreover, for $i\in Q_0$ and $c\geq0$, we will use the following notation:
\begin{align*}
& B(Q,\bfm;\infty)_{\bfr} := B(Q,\bfm;\infty)\left[-\sum_{i\in Q_0}r_i\alpha_i\right] , \\
& B(Q,\bfm;\infty)_{\bfr,c\bfalpha_i} :=B(Q,\bfm;\infty)\left[-\sum_{i\in Q_0}r_i\alpha_i\right]_{i,c}.
\end{align*}
\end{notation}

\begin{remark}
Note that Geiss--Leclerc--Schr\"{o}er construct the lowest-weight crystal associated to the \emph{upper} half of the quantised enveloping algebra of $\mathfrak{g}_{Q,\bfm}$ \cite[\S 5]{GLS18a}. One obtains instead the highest-weight crystal $B(C;\infty)$ as introduced in \cref{Exmp/B(infty)} using the \emph{lower} half of the quantised enveloping algebra by exchanging raising and lowering operators.
\end{remark}

The main result of \cite{GLS18a} is the following.

\begin{theorem}\label{Thm/Geiss-Leclerc-SchröerCrystal}\emph{\cite[Thm.\ 5.4]{GLS18a}}
Let $(Q,\bfm)$ be a quiver with multiplicities without loops and let $C$ be the associated symmetrisable Cartan matrix and $\mathfrak{g}_{Q,\bfm}$ be the corresponding Kac--Moody algebra. Then there is an isomorphism of $\mathfrak{g}_{Q,\bfm}$-crystals:
\[
B(Q,\bfm;\infty)\simeq B(C;\infty)
.
\]
\end{theorem}

We finish by defining a $\mathfrak{g}_{Q,\bfm}$-crystal for framed quivers associated to $(Q,\bfm)$. This will be used in \cref{Sect/selectCrystalComponents} in our construction of $B(C;\lambda)$.

\begin{definition}
For a quiver with multiplicities $(Q,\bfm)$ and framing vector $\bff\in\bbZ_{\geq0}^{Q_0}$, let
\[
B(Q,\bff,\bfm;\infty):=\bigsqcup_{\bfr\in\bbZ_{\geq0}^{Q_0}}B(Q_\bff,\widehat{\bfm};\infty)_{\widehat{\bfr}}
\]
endowed with the $\mathfrak{g}_{Q,\bfm}$-crystal structure given by the crystal operators $e_i,f_i$ for $i\in Q_0$ of $B(Q_\bff,\widehat{\bfm};\infty)$, the weight function $\mathrm{wt}(B(Q_\bff,\widehat{\bfm};\infty)_{\widehat{\bfr}}):=-\sum_{i\in Q_0}r_i\alpha_i$ and $\epsilon_i$ for $ i\in Q_0$ obtained by restriction from $B(Q_\bff,\widehat{\bfm};\infty)$ (the $\phi_i$ are then determined by \ref{C1}).
\end{definition}

\begin{notation}
In what follows, given $\bfr\in\bbZ_{\geq0}^{Q_0}$, we will write for short:
\[
B(Q,\bff,\bfm;\infty)_\bfr:=B(Q_\bff,\widehat{\bfm};\infty)_{\widehat{\bfr}}.
\]
We denote by $b_{\bff,0}$ the unique element of $B(Q,\bff,\bfm;\infty)_{\mathbf{0}}$.
\end{notation}

\subsection{Selecting crystal components}\label{Sect/selectCrystalComponents}

In this section, we adapt Saito's construction of $B(C;\lambda)$ from the symmetric case to the symmetrisable case \cite{Sai02}. We proceed in a very similar way, by relating a certain subset of irreducible components of $\Lambda_{Q_\bff,\widehat{\bfm}}$ to irreducible components of $\Lambda_{Q,\bfm}$. These components will be selected by the stability condition defined in \cref{Sect/constructionNilpQuiverVar}. Here, we extract from \cite{Sai02} the abstract requirements that stable crystal components should meet for Saito's argument to work. We will check that these properties are satisfied in \cref{Sect/GeometricHWCrystals}.

From now on, we fix a quiver with multiplicities $(Q,\bfm)$ and a framing vector $\bff\in\bbZ_{\geq0}^{Q_0}$. Let $\lambda:=\sum_{i\in Q_0}f_i\varpi_i$ denote the corresponding weight. Below we will describe which irreducible components of $\Lambda_{Q_\bff,\widehat{\bfm}}$ are relevant for the geometric construction of the crystal $B(C;\lambda)$, but first we show that any subset has an induced crystal structure.

\begin{proposition}\label{Prop/crystalRestriction}
Any subset $B\subset B(Q,\bff,\bfm;\infty)\otimes T(\lambda)$ has an induced crystal structure defined by restricting the functions $\mathrm{wt},\phi_i$ and $\epsilon_i$ from $B(Q,\bff,\bfm;\infty)\otimes T(\lambda)$ to $B$ and defining restricted raising and lowering operators as follows
\[
\begin{array}{rrcl}
e_i\vert_B\colon  & B & \rightarrow & B\sqcup\{0\} \\
& b & \mapsto & 
\left\{
\begin{array}{ll}
e_i(b) & \text{if } e_i(b)\in B, \\
0 & \text{else.}
\end{array}
\right.
\end{array}
\begin{array}{rrcl}
f_i\vert_B \colon & B & \rightarrow & B\sqcup\{0\} \\
& b & \mapsto &
\left\{
\begin{array}{ll}
f_i(b) & \text{if } f_i(b)\in B, \\
0 & \text{else.}
\end{array}
\right.
\end{array}
\]

\end{proposition}

\begin{proof}
The defining axioms of crystals can be checked immediately.
\end{proof}

\begin{notation}\label{notation for weight shift}
From now on, we write $e_i,f_i$ for $e_i\vert_B,f_i\vert_B$ for $i\in Q_0$. Moreover, for $i\in Q_0$ and $c\geq0$, we will use the following notation:
\begin{align*}
& B_{\bfr} := B\left[\lambda-\sum_{i\in Q_0}r_i\alpha_i\right] , \\
& B_{\bfr,c\bfalpha_i} :=B\left[\lambda-\sum_{i\in Q_0}r_i\alpha_i\right]_{i,c}.
\end{align*}
Note that for $b \in B_\bfr$, we have $\mathrm{wt}(b) = \lambda - \sum_{i \in Q_0} r_i \alpha_i$.
\end{notation}

Under the following conditions, $B$ is a highest-weight crystal of highest weight $\lambda$, which allows us to use \cref{Lem/HWcrystalsF-operators}. We will denote by $b_\lambda\in B$ the highest-weight element.

\begin{proposition}\label{Prop/HWcrystalRestriction}
Let $B$ be a subset of $B(Q,\bff,\bfm;\infty)\otimes T(\lambda)$ with its induced crystal structure as in \cref{Prop/crystalRestriction}. If:
\begin{enumerate}[label=(\roman*)]
\item $b_{\bff,0}\in B$,
\item for every $b\in B\setminus\{b_{\bff,0}\}$, there exists $i\in Q_0$ such that $\epsilon_i(b)>0$,
\item\label{Item3/Prop/HWcrystalRestriction} for all $i\in Q_0$ and $c\in\bbZ_{\geq0}$, for all $\bfr\in\bbZ_{\geq0}^{Q_0}$, whenever $B_{\bfr,c\bfalpha_i}\ne\emptyset$, we have:
\[
e_i^c(B_{\bfr,c\bfalpha_i})=B_{\bfr-c\bfalpha_i,0\bfalpha_i},
\]
\end{enumerate}
then $B$ is a highest-weight crystal of highest weight $\lambda$.
\end{proposition}

\begin{proof}
We prove the two defining properties from \cref{def/highestWeightCrystal} for $b\in B_\bfr$. We prove the first property by induction on $\vert\bfr\vert:=\sum_{i\in Q_0}r_i$. The case where $\bfr=\mathbf{0}$ is satisfied by assumption (i). Consider now $b\in B_\bfr$, where $\bfr>\mathbf{0}$. By assumption (ii), there exists $i\in Q_0$ such that $\epsilon_i(b)>0$, thus, by assumption (iii), there exists $b'\in B_{\bfr-\bfalpha_i}$ such that $b=f_i(b')$. The first defining property then follows from the induction hypothesis, as $\vert\bfr-\bfalpha_i\vert<\vert\bfr\vert$. Regarding the second property, given $b\in B_\bfr$ and $i\in Q_0$, assumption (iii) yields $b':=e_i^c(b)\in B_{\bfr-c\bfalpha_i,0\bfalpha_i}$, where $c:=\epsilon_i(b)$. Since, by \cref{Thm/Geiss-Leclerc-SchröerCrystal}, $e_i(B(Q,\bff,\bfm;\infty)_{\bfr-c\bfalpha_i,0\bfalpha_i})=0$, we obtain that $b$ satisfies the second defining property, which finishes the proof.
\end{proof}

The following proposition closely follows Saito's geometric construction of $B(C;\lambda)$ in the symmetric setting \cite{Sai02}.

\begin{proposition}\label{Prop/geometricCriteriaB(lambda)}
Suppose that the following hold:
\begin{enumerate}[label=(\arabic*)]
\item\label{Item1/Prop/geometricCriteriaB(lambda)} $b_{\bff,0}\in B$;
\item\label{Item2/Prop/geometricCriteriaB(lambda)} for every $b\in B\setminus\{b_{\bff,0}\}$, there exists $i\in Q_0$ such that $\epsilon_i(b)>0$;
\item\label{Item3/Prop/geometricCriteriaB(lambda)} for every $b\in B$, for all $(x,y,s,t)\in Z_b$, we have $s=0$;
\item\label{Item4/Prop/geometricCriteriaB(lambda)} for all $i\in Q_0$, $\bfr\in\bbZ_{\geq0}^{Q_0}$ and $c\in\bbZ_{\geq0}$ such that $B_{\bfr,c\bfalpha_i}\ne\emptyset$, for all $l\in\bbZ$, we have:
\[
B_{\bfr+l\bfalpha_i,(c+l)\bfalpha_i}\ne\emptyset
\Leftrightarrow
-c\leq l\leq c+\left\langle h_i,\lambda-\sum_{j\in Q_0}r_j\alpha_j\right\rangle
;
\]
when this is the case, we further have $B_{\bfr+l\bfalpha_i,(c+l)\bfalpha_i}=f_i^l(B_{\bfr,c\bfalpha_i})$ if $l\geq0$ (resp.\ $B_{\bfr+l\bfalpha_i,(c+l)\bfalpha_i}=e_i^{-l}(B_{\bfr,c\bfalpha_i})$ if $l<0$).
\end{enumerate}
Then we have an isomorphism of crystals $B\simeq B(C;\lambda)$.
\end{proposition}

\begin{proof}
We check the criteria in \cref{Prop/SaitoCriterion}. Note that the first two assumptions coincide with those of \cref{Prop/HWcrystalRestriction}, and Assumption \ref{Item4/Prop/geometricCriteriaB(lambda)} implies Assumption \ref{Item3/Prop/HWcrystalRestriction} of \cref{Prop/HWcrystalRestriction}, so $B$ is a highest-weight crystal of highest weight $\lambda$.

Let us now construct a lower semistrict morphism of crystals
\[
\Phi: B(C;\infty)\otimes T(\lambda)\sqcup\{0\}\rightarrow B\sqcup\{0\}
\]
such that: i) $\Phi(b(\infty)\otimes t(\lambda))=b_\lambda$, ii) $\mathrm{Im}(\Phi)=B\sqcup\{0\}$ and iii) $\Phi$ induces a bijection:
\[
\{b\in B(C;\infty)\otimes T(\lambda)\ \vert\ \Phi(b)\ne0\}\overset{\sim}{\rightarrow} B.
\]
We first construct an injection of sets $\iota \colon B\hookrightarrow B(Q,\bfm;\infty)\otimes T(\lambda)$. Consider the subscheme:
\begin{equation}\label{Def/Lambdatilde}    
\tilde{\Lambda}_{Q_\bff,\widehat{\bfm}}=\{ (x,y,s,t)\in R(\overline{Q_\bff},\widehat{\bfm})\ \vert\ (x,y)\in\Lambda_{Q,\bfm}\text{ and } s=0\}
\subset
\Lambda_{Q_\bff,\widehat{\bfm}}
.
\end{equation}
The trivial affine fibration
\[
\mathrm{pr}:\mu_{Q_\bff,\widehat{\bfm}}^{-1}(0)\cap\{s=0\}\rightarrow\mu_{Q,\bfm}^{-1}(0)
\]
(forgetting framing maps) restricts to a trivial affine fibration $\tilde{\Lambda}_{Q_\bff,\widehat{\bfm}}\rightarrow\Lambda_{Q,\bfm}$ of (local) dimension $\frac{1}{2}(\dim R(\overline{Q_\bff},\widehat{\bfm})-\dim R(\overline{Q},\bfm))$. Indeed, by definition, $\mathrm{pr}(\tilde{\Lambda}_{Q_\bff,\widehat{\bfm}})\subset\Lambda_{Q,\bfm}$. Conversely, we have 
\[
\mathrm{pr}^{-1}(\Lambda_{Q,\bfm}\cap\mathrm{nil}_{Q,\bfm})\subset\tilde{\Lambda}_{Q_\bff,\widehat{\bfm}}\cap\mathrm{nil}_{Q_\bff,\bfm}.
\]
The left-hand side is contained in $\mathrm{nil}_{Q_\bff,\bfm}\cap\{s=0\}$ by construction and in $\Lambda_{Q_\bff,\widehat{\bfm}}$ by \cref{Prop/dimCrystalComponents}, as it is equidimensional of (local) dimension $\dim\Lambda_{Q_\bff,\widehat{\bfm}}$. Since $\mathrm{pr}$ is open, we obtain
\[
\mathrm{pr}^{-1}(\Lambda_{Q,\bfm})\subset\tilde{\Lambda}_{Q_\bff,\widehat{\bfm}}
\]
by taking closures. This yields the required trivial affine fibration $\tilde{\Lambda}_{Q_\bff,\widehat{\bfm}}\rightarrow\Lambda_{Q,\bfm}$, hence a bijection between irreducible components of $\tilde{\Lambda}_{Q_\bff,\widehat{\bfm}}$ and $\Lambda_{Q,\bfm}$. By Assumption  \ref{Item3/Prop/geometricCriteriaB(lambda)}, for every $b\in B$, the closed crystal component $Z_b$ is contained in $\tilde{\Lambda}_{Q_\bff,\widehat{\bfm}}$. This gives the required injection $\iota \colon B\hookrightarrow B(Q,\bfm;\infty)\otimes T(\lambda)$.

We now define $\Phi$. Consider the following map of sets:
\[
\begin{array}{rcl}
B(Q,\bfm;\infty)\otimes T(\lambda) & \rightarrow & B\sqcup\{0\} \\
b\otimes t(\lambda) & \mapsto &
\left\{
\begin{array}{ll}
\iota^{-1}(b) & \text{if } b\in\mathrm{Im}(\iota), \\
0 & \text{else.}
\end{array}
\right.
\end{array}
\]
We define $\Phi$ to be the composition of the above map with the isomorphism of crystals $B(C;\infty)\otimes T(\lambda)\simeq B(Q,\bfm;\infty)\otimes T(\lambda)$ (see \cref{Thm/Geiss-Leclerc-SchröerCrystal}).

By \cref{Lem/forgetfulCrystalMap} below, $\Phi$ is a lower semistrict morphism of crystals. We further claim that $\Phi$ satisfies i), ii) and iii). Part i) of the claim holds by construction, as $b(\infty)$ (resp.\ $b_\lambda$) corresponds to the unique irreducible component of $\Lambda_{Q,\bfm;\mathbf{0}}$ (resp.\ $\Lambda_{Q_\bff,\widehat{\bfm};\mathbf{0}}$). Parts ii) and iii) hold by construction and because $\iota$ provides an inverse.

Finally, let us check that, for all $i\in Q_0$ and $b\in B$, we have:
\begin{align*}
& \phi_i(b)=\max\{ k\geq0\ \vert\ f_i^k(b)\ne0\} , \\
& \epsilon_i(b)=\max\{ k\geq0\ \vert\ e_i^k(b)\ne0\} .
\end{align*}
The above property for $\epsilon_i$ ($i\in Q_0$) holds because $B$ is a highest-weight crystal, by \cref{Prop/HWcrystalRestriction}. Let us now check it for $\phi_i$ ($i\in Q_0$). Firstly, Assumption  \ref{Item4/Prop/geometricCriteriaB(lambda)}  implies that, for all $\bfr\in\bbZ_{\geq0}^{Q_0}$, $i\in Q_0$, $c,k\in\bbZ_{\geq0}$, and for every $b\in B_{\bfr,c\bfalpha_i}$, we have:
\[
f_i^k(b)\ne0
\Leftrightarrow
B_{\bfr+k\bfalpha_i,(c+k)\bfalpha_i}\ne\emptyset
.
\]
Secondly, given $b\in B_{\bfr,c\bfalpha_i}$ (so that $\epsilon_i(b)=c$), Assumption  \ref{Item4/Prop/geometricCriteriaB(lambda)} also gives:
\[
B_{\bfr+k\bfalpha_i,(c+k)\bfalpha_i}\ne\emptyset
\Leftrightarrow
k\leq c+\left\langle h_i,\lambda-\sum_{j\in Q_0}r_j\alpha_j\right\rangle.
\]
Since $c+\langle h_i,\lambda-\sum_{j\in Q_0}r_j\alpha_j\rangle=\epsilon_i(b)+\langle h_i,\mathrm{wt}(b)\rangle =\phi_i(b)$ (recall the weights are shifted by $\lambda$, see \cref{notation for weight shift}), we obtain that $\phi_i(b)=\max\{ k\geq0\ \vert\ f_i^k(b)\ne0\}$, as required. This finishes the proof.
\end{proof}

\begin{lemma}\label{Lem/forgetfulCrystalMap}
Let $\Phi \colon B(C;\infty)\otimes T(\lambda)\sqcup\{0\}\rightarrow B\sqcup\{0\}$ be the map constructed above. Then $\Phi$ is a lower semistrict morphism of crystals.
\end{lemma}

\begin{proof}
Recall the map $\iota \colon B\hookrightarrow B(Q,\bfm;\infty)\otimes T(\lambda)$ constructed above. We first show that $\iota$ commutes with the lowering operators $f_i$ for $ i\in Q_0$. More precisely, for all $i\in Q_0$, for all $b,b'\in B$, we have $b'=f_i(b)\Rightarrow \iota(b')=f_i(\iota(b))$.

Let $b,b'\in B$ and $i\in Q_0$ such that $b'=f_i(b)$. Set $c:=\epsilon_i(b)$ and $\bfr\in\bbZ_{\geq0}^{Q_0}$ such that $\mathrm{wt}(b)=\lambda-\sum_{i\in Q_0}r_i\alpha_i$. For the sake of conciseness, let us write $\Lambda:=\Lambda_{Q_\bff,\widehat{\bfm}}$, $\tilde{\Lambda}:=\tilde{\Lambda}_{Q_\bff,\widehat{\bfm}}$ as in \eqref{Def/Lambdatilde}  and $\Lambda':=\Lambda_{Q,\bfm}$.
Consider the following diagram:
\[
\begin{tikzcd}
\Lambda_{\widehat{\bfr},c\bfalpha_i} & \Lambda_{\widehat{\bfr},\widehat{\bfr}+\bfalpha_i;(c+1)\bfalpha_i}\ar[l,"q", swap]\ar[r,"p"] & \Lambda_{\widehat{\bfr}+\bfalpha_i,(c+1)\bfalpha_i} \\
\tilde{\Lambda}_{\widehat{\bfr},c\bfalpha_i} \ar[u,hookrightarrow]\ar[d,"\pi_\bfr"] & \tilde{\Lambda}_{\widehat{\bfr},\widehat{\bfr}+\bfalpha_i;(c+1)\bfalpha_i} \ar[u,hookrightarrow]\ar[d]\ar[l,"\tilde{q}",swap]\ar[r,"\tilde{p}"]\arrow[ul, phantom, "\ulcorner", very near start]\arrow[ur, phantom, "\urcorner", very near start] & \tilde{\Lambda}_{\widehat{\bfr}+\bfalpha_i,(c+1)\bfalpha_i} \ar[u,hookrightarrow]\ar[d,"\pi_{\bfr+\bfalpha_i}"] \\
\Lambda'_{\bfr,c\bfalpha_i} & \Lambda'_{\bfr,\bfr+\bfalpha_i;(c+1)\bfalpha_i}\ar[l,"q'",swap]\ar[r,"p'"] & \Lambda'_{\bfr+\bfalpha_i,(c+1)\bfalpha_i}
\end{tikzcd}
,
\]
where $\tilde{\Lambda}_{\widehat{\bfr},c\bfalpha_i}$ (resp. $\tilde{\Lambda}_{\widehat{\bfr}+\bfalpha_i,(c+1)\bfalpha_i}$) is the union of irreducible components of $\tilde{\Lambda}_{\widehat{\bfr}}\cap\Lambda_{\widehat{\bfr},c\bfalpha_i}$ (resp. $\tilde{\Lambda}_{\widehat{\bfr}+\bfalpha_i}\cap\Lambda_{\widehat{\bfr}+\bfalpha_i,(c+1)\bfalpha_i}$) of dimension $\dim\Lambda_{\widehat{\bfr}}$ (resp. $\dim\Lambda_{\widehat{\bfr}+\bfalpha_i}$), as in \cref{def/cokernelLoci}.

In the above diagram, the horizontal arrows induce bijections between sets of irreducible components (see \cite[Lem.\ 3.6]{GLS18a}). The vertical arrows between the top and the middle row are closed immersions between schemes of equal dimensions, so they induce injections between sets of top-dimensional irreducible components. Finally, the vertical arrows between the middle and the bottom row are trivial affine fibrations of suitable relative dimension, so that they induce bijections between sets of top-dimensional irreducible components.

Let $Z_b,Z_{b'}$ be the irreducible components of $\Lambda$ corresponding to $b,b'$. By definition, we have:
\[
Z_{b'}\cap\Lambda_{\widehat{\bfr}+\bfalpha_i,(c+1)\bfalpha_i}=p(q^{-1}(Z_b\cap\Lambda_{\widehat{\bfr},c\bfalpha_i})).
\]
Since $Z_b,Z_{b'}$ are contained in $\tilde{\Lambda}$, we obtain:
\[
Z_{b'}\cap\tilde{\Lambda}_{\widehat{\bfr}+\bfalpha_i,(c+1)\bfalpha_i}=\tilde{p}(\tilde{q}^{-1}(Z_b\cap\tilde{\Lambda}_{\widehat{\bfr},c\bfalpha_i})).
\]
Commutativity of the bottom squares then implies 
\[
\pi_{\bfr+\bfalpha_i}(Z_{b'})\cap\Lambda'_{\bfr+\bfalpha_i,(c+1)\bfalpha_i}=p'((q')^{-1}(\pi_{\bfr}(Z_b)\cap\Lambda'_{\bfr,c\bfalpha_i})),
\]
hence $\iota(b')=f_i(\iota(b))$. This proves that $\iota$ commutes with lowering operators.

Let us now check that $\Phi$ is a semistrict morphism of crystals. First observe that $\Phi$ preserves the functions $\epsilon_i$ for $i\in Q_0$, as by Assumption \ref{Item3/Prop/geometricCriteriaB(lambda)}, for any $b\in B$, the quantity $\rk\fac_i(M(x,y,s=0,t))$ (for a general $(x,y,s=0,t)\in Z_b$) is preserved under forgetting the framing maps. $\Phi$ also preserves the function $\mathrm{wt}$, since $\iota$ preserves rank vectors (in restriction to $Q_0$). Then $\Phi$ preserves functions $\phi_i,\ i\in Q_0$ as well by \cref{Rem/crystalOperatorsE/F}. It remains to show that $\Phi$ always commutes with the lowering operators $f_i,\ i\in Q_0$, as the required commuting with raising operators $e_i,\ i\in Q_0$ then follows from \cref{Rem/crystalOperatorsE/F}. Since $\iota$ commutes with lowering operators, we obtain that $f_i(\Phi(b))=\Phi(f_i(b))$ for $b\in B(C;\infty)\otimes T(\lambda)$ and $i\in Q_0$, whenever $\Phi(b)\ne0$ and $f_i(\Phi(b))\ne0$. In order to recover the remaining cases, it is enough to show that for any $i\in Q_0$ and $b\in B(C;\infty)\otimes T(\lambda)$, we have $f_i(\Phi(b))=0$ if and only if $\Phi(f_i(b))=0$. Recall that, since $B(C;\infty)\otimes T(\lambda)$ and $B$ are both highest-weight crystals, we have
\[
\epsilon_i(b)=\max\{k\geq 0\ \vert\ e_i^k(b)\ne0\}
\]
for every $i\in Q_0$ and $b\in B(C;\infty)\otimes T(\lambda)$ (resp. $b\in B$). Suppose that $\Phi(f_i(b))\ne0$. Then $f_i(b)\ne0$ and $b\ne0$, so $\epsilon_i(\Phi(f_i(b)))=\epsilon_i(f_i(b))>0$ and $e_i(\Phi(f_i(b)))\ne0$. By surjectivity of $\Phi$, there exists $b'\in B(C;\infty)\otimes T(\lambda)$ such that $\Phi(b')=e_i(\Phi(f_i(b)))$. Since $\Phi(b')\ne0$ and $f_i(\Phi(b'))=\Phi(f_i(b))\ne0$, we obtain $\Phi(f_i(b'))=f_i(\Phi(b'))=\Phi(f_i(b))$, hence $b=b'$ and $f_i(\Phi(b))=\Phi(f_i(b))\ne0$. Conversely, suppose that $f_i(\Phi(b))\ne0$. Then $\Phi(b)\ne0$, so $f_i(b)=f_i(\iota(\Phi(b)))=\iota(f_i(\Phi(b)))$, hence $\Phi(f_i(b))\ne0$ by definition. 
\end{proof}

\section{Geometric construction of crystals}\label{Sect/GeometricHWCrystals}

In this section, we complete the geometric construction of the crystal $B(C;\lambda)$ using the stability condition defined in \cref{Sect/constructionNilpQuiverVar}. We keep working with the framed quiver $(Q_\bff,\widehat{\bfm})$, such that $\lambda=\sum_{i\in Q_0}f_i\varpi_i$, and continue to assume that $Q$ has no loop arrows. To prove \cref{MainThm/crystalConstruction}, it remains to check that the stable irreducible components of $\Lambda_{Q_\bff,\widehat{\bfm}}$ satisfy the requirements of \cref{Prop/geometricCriteriaB(lambda)}.

We start with the definition of the subcrystal $B_\bff^\st\subset B(Q,\bff,\bfm;\infty)$, which we will identify with $B(C;\lambda)$.

\begin{definition}\label{def/geometricCrystalB(lambda)}
Let $B^\st=B_\bff^\st\subset B(Q,\bff,\bfm;\infty)$ be the subset of irreducible components $Z\subset\Lambda$ such that $Z^\st\ne\emptyset$, which we endow with the induced crystal structure from \cref{Prop/crystalRestriction}.
\end{definition}

We will check that $B_\bff^\st\subset B(Q,\bff,\bfm;\infty)$ satisfies the assumptions in \cref{Prop/geometricCriteriaB(lambda)}. Assumption \ref{Item3/Prop/geometricCriteriaB(lambda)} holds by \cref{Lem/framingVanishing} (note that the closed condition $s=0$ holds on a dense subset of $\Lambda^\st$), and we next check Assumption \ref{Item2/Prop/geometricCriteriaB(lambda)}.

\begin{proposition}\label{Prop/cokernelsStableRep}
Let $\bfr\in\bbZ_{\geq0}^{Q_0}\setminus\{0\}$ and $M$ be a stable crystal representation of $(\overline{Q_\bff},\widehat{\bfm})$ with rank $\widehat{\bfr}$. Then there exists $i\in Q_0$ such that $\rk \fac_i(M)>0$.
\end{proposition}

\begin{proof}
Since $M$ is E-filtered, there exists $i\in Q_0\sqcup\{\infty\}$ such that $\rk \fac_i(M)>0$. By contradiction, if $i=\infty$, then there exists a subrepresentation $M'\subset M$ such that $M/M'\simeq \fac_i(M)$ is concentrated at vertex $\infty$. This contradicts stability, as $M'\ne0$.
\end{proof}

We now check Assumption \ref{Item4/Prop/geometricCriteriaB(lambda)} in \cref{Prop/geometricCriteriaB(lambda)}.

\begin{proposition}\label{Prop/crystalOperatorRange}
Let $i\in Q_0$ and $c\in\bbZ_{\geq0}$. Suppose that $\Lambda_{\widehat{\bfr},c\bfalpha_i}^\st\ne\emptyset$. Then, for $l\in\bbZ$, we have:
\[
\Lambda_{\widehat{\bfr}+l\bfalpha_i,(c+l)\bfalpha_i}^\st\ne\emptyset
\Leftrightarrow
-c\leq l\leq c+\left\langle h_i,\lambda-\sum_{j\in Q_0}r_j\alpha_j\right\rangle
.
\]
Moreover, when this is the case, we further have $B_{\bfr+l\bfalpha_i,(c+l)\bfalpha_i}^\st=f_i^l(B_{\bfr,c\bfalpha_i}^\st)$ if $l\geq0$ (resp.\ $B_{\bfr+l\bfalpha_i,(c+l)\bfalpha_i}^\st=e_i^{-l}(B_{\bfr,c\bfalpha_i}^\st)$ if $l<0$).
\end{proposition}

\begin{proof} To avoid any possible confusion between the lowering operators $f_i$ and the framing vector $\bff$, we shall write $\bff = (\lambda_i)_{i \in Q_0}$ in this proof. 

Let us prove the forwards implication in the first assertion. Suppose $\Lambda_{\widehat{\bfr}+l\bfalpha_i,(c+l)\bfalpha_i}^\st\ne\emptyset$. For $(x,y,s,t)\in\Lambda_{\widehat{\bfr}+l\bfalpha_i,(c+l)\bfalpha_i}^{\mathrm{cr},\st}\ne\emptyset$, consider $M=M(x,y,s,t)$ and the following linear maps:
\[
M_i
\overset{M_{i,\mathrm{out}}}{\rightarrow}
\bigoplus_{a:i\rightarrow j}k_a\otimes_{k_{m_j}}M_j
\overset{\widetilde{M}_{i,\mathrm{in}}}{\rightarrow}
M_i
.
\]
By \cref{def/GeissLeclercSchröerCrystal}, we have $c+l=\rk\Coker(\widetilde{M}_{i,\mathrm{in}})\geq0$. The moment map equation at vertex $i$ reads $\widetilde{M}_{i,\mathrm{in}}\circ M_{i,\mathrm{out}}=0$. Moreover, by stability, $M_{i,\mathrm{out}}$ must be injective (see \cref{Lem/stabilityCriterion}). Hence $M_i \simeq \mathrm{Im}(M_{i,\mathrm{out}}) \subset \ker(\widetilde{M}_{i,\mathrm{in}})$ and we have
\[
\rk\left(\bigoplus_{a:i\rightarrow j}k_a\otimes_{k_{m_j}}M_j\right)
=
\rk M_i-\rk\Coker(\widetilde{M}_{i,\mathrm{in}})+\rk\Ker(\widetilde{M}_{i,\mathrm{in}})
\geq
2(r_i+l)-(c+l).
\]
We can express this in terms of the pairing using \eqref{pairings/rootscoroots} and $\lambda=\sum_{i\in Q_0}\lambda_i\varpi_i$ as
\[
l\leq c+\lambda_i-\left(2r_i+\sum_{i\ne j}c_{ij}r_j\right)=c+\left\langle h_i,\lambda-\sum_{j\in Q_0}r_j\alpha_j\right\rangle
.
\]

Conversely, assume that $-c\leq l\leq c+\langle h_i,\lambda-\sum_{j\in Q_0}r_j\alpha_j\rangle$. We directly prove the second assertion, which implies the backwards implication of the first assertion, as we assumed $\Lambda_{\widehat{\bfr},c\bfalpha_i}^\st\ne\emptyset$. Consider $M=M(x,y,s,t)$ for some $(x,y,s,t)\in\Lambda_{\widehat{\bfr},c\bfalpha_i}^{\mathrm{cr},\st}$. We start by constructing a point in $\mathrm{nil}_{\widehat{\bfr}+l\bfalpha_i,(c+l)\bfalpha_i}^{\st}$ from $M$.

Suppose first that $-c\leq l\leq 0$. By assumption, $\rk\Coker(M_{i,\mathrm{in}})=c$. Therefore, we may choose a free $k_{m_i}$-submodule $M'_i\subset M_i$ of rank $r_i+l$, which contains $\Ima(\widetilde{M}_{i,\mathrm{in}})$. Setting $M'_j=M_j$ for $j\ne i$, we obtain a subrepresentation $M'\subset M$. By \cref{Lem/stabilityHeckeModification}, $M'$ is also stable. Moreover, $M'$ is $E$-filtered and satisfies the moment map equation $\widetilde{M}'_{i,\mathrm{in}}\circ M'_{i,\mathrm{out}}=0$, by construction. Thus $M'$ corresponds to an orbit in $\mathrm{nil}_{\widehat{\bfr}+l\bfalpha_i,(c+l)\bfalpha_i}^{\st}$, as required.

Suppose now that $0\leq l\leq c+\langle h_i,\lambda-\sum_{j\in Q_0}r_j\alpha_j\rangle$. By assumption, $\rk\Ker(M_{i,\mathrm{in}})=r_i+c+\langle h_i,\lambda-\sum_{j\in Q_0}r_j\alpha_j\rangle$. We may thus choose a free $k_{m_i}$-submodule $M_i'$ of rank $r_i+l$:
\[
M_i\overset{M_{i,\mathrm{out}}}{\hookrightarrow}M'_i\subset\Ker(\widetilde{M}_{i,\mathrm{in}})
\]
and as above let $M'$ be the representation of rank $\widehat{\bfr}+l\bfalpha_i$ containing $M$ with $M'_j = M_j$ for $j \neq i$, with:
\begin{align*}
& M'_{i,\mathrm{out}}\colon M'_i\subset\bigoplus_{a:i\rightarrow j}k_a\otimes_{k_{m_j}}M_j=\bigoplus_{a:i\rightarrow j}k_a\otimes_{k_{m_j}}M'_j, \\
& \widetilde{M}'_{i,\mathrm{in}}\colon\bigoplus_{a:j\rightarrow i}k_a\otimes_{k_{m_j}}M'_j=\bigoplus_{a:j\rightarrow i}k_a\otimes_{k_{m_j}}M_j\overset{\widetilde{M}_{i,\mathrm{in}}}{\rightarrow}M_i\overset{M_{i,\mathrm{out}}}{\hookrightarrow}M'_i .
\end{align*}
By construction, $M'$ is E-filtered and satisfies the moment map equations $\widetilde{M}'_{i,\mathrm{in}}\circ M'_{i,\mathrm{out}}=0$. Moreover, $M'_{i,\mathrm{out}}$ is injective and $M$ is stable, hence $M'$ is stable by \cref{Lem/stabilityHeckeModification}.

Now consider the following diagram:
\[
\Lambda_{\widehat{\bfr},c\bfalpha_i}\overset{q}{\longleftarrow}\Lambda_{\widehat{\bfr},\widehat{\bfr}+l\bfalpha_i;(c+l)\bfalpha_i}\overset{p}{\longrightarrow}\Lambda_{\widehat{\bfr}+l\bfalpha_i,(c+l)\bfalpha_i}
\text{ if }l\geq0
\]
\[
\text{(resp.\ }
\Lambda_{\widehat{\bfr}+l\bfalpha_i,(c+l)\bfalpha_i}\overset{q}{\longleftarrow}\Lambda_{\widehat{\bfr}+l\bfalpha_i,\widehat{\bfr};c\bfalpha_i}\overset{p}{\longrightarrow}\Lambda_{\widehat{\bfr},c\bfalpha_i}
\text{ if }l<0\text{).}
\]

Let us show that $f_i^l(B_{\bfr,c\bfalpha_i}^\st)\subset B_{\bfr+l\bfalpha_i,(c+l)\bfalpha_i}^\st$ if $l\geq0$ (resp.\ $e_i^{-l}(B_{\bfr,c\bfalpha_i}^\st)\subset B_{\bfr+l\bfalpha_i,(c+l)\bfalpha_i}^\st$ if $l<0$). 
For $b\in B_{\bfr,c\bfalpha_i}^\st$, consider $b'=f_i^l(b)\in B(Q,\bff,\bfm;\infty)$ if $l\geq0$ (resp.\ $b'=e_i^{-l}(b)\in B(Q,\bff,\bfm;\infty)$ if $l<0$). Note that $b'\ne0$, since the operator $f_i$ does not vanish on $B(C;\infty)$ (resp.\ $-l\leq\epsilon_i(b)=c$); see \cref{Exmp/B(infty)} and \cref{Thm/Geiss-Leclerc-SchröerCrystal}. Let $Z_b$ and $Z_{b'}$ denote the associated irreducible components of $\Lambda$. By \cite[Lem.\ 3.9]{GLS18a}, we have:
\[
Z_{b'}\cap\Lambda_{\widehat{\bfr}+l\bfalpha_i,(c+l)\bfalpha_i}=p(q^{-1}(Z_b\cap\Lambda_{\widehat{\bfr},c\bfalpha_i}))
\text{ if } l\geq0
\]
\[
\text{(resp. }
Z_{b'}\cap\Lambda_{\widehat{\bfr}+l\bfalpha_i,(c+l)\bfalpha_i}=q(p^{-1}(Z_b\cap\Lambda_{\widehat{\bfr},c\bfalpha_i}))
\text{ if } l<0 \text{).}
\]
Moreover, for a general stable point $(x,y,s,t)\in Z_b$, general points in $p(q^{-1}(x,y,s,t))$ (resp.\ $q(p^{-1}(x,y,s,t))$) correspond to stable representations $M'\supset M(x,y,s,t)$ if $l\geq0$ (resp.\ stable subrepresentations $M'\subset M(x,y,s,t)$ if $l<0$), as constructed above. Thus, a general point in $Z_{b'}$ is stable (and crystal), hence $b'\in B_{\bfr+l\bfalpha_i,(c+l)\bfalpha_i}^\st$. In particular, $\Lambda_{\widehat{\bfr}+l\bfalpha_i,(c+l)\bfalpha_i}^{\mathrm{cr},\st}\ne\emptyset$, as the previous construction generically yields crystal representations.

Let us now show that $B_{\bfr,c\bfalpha_i}^\st$ surjects onto $B_{\bfr+l\bfalpha_i,(c+l)\bfalpha_i}^\st$. Let $b'\in B_{\bfr+l\bfalpha_i,(c+l)\bfalpha_i}^\st$ and $b=e_i^l(b')\in B(Q,\bff,\bfm;\infty)$ if $l\geq0$ (resp.\ $b=f_i^{-l}(b')\in B(Q,\bff,\bfm;\infty)$ if $l<0$). Note that $b\ne 0$, as $l\leq\epsilon_i(b')=c+l$ (resp.\ the operator $f_i$ does not vanish on $B(C;\infty)$); see \cref{Exmp/B(infty)} and \cref{Thm/Geiss-Leclerc-SchröerCrystal}. By \cite[Lem.\ 3.9]{GLS18a}, we have:
\[
Z_b\cap\Lambda_{\widehat{\bfr},c\bfalpha_i}=q(p^{-1}(Z_{b'}\cap\Lambda_{\widehat{\bfr}+l\bfalpha_i,(c+l)\bfalpha_i})) \text{ if } l\geq0
\]
\[
\text{(resp. } Z_b\cap\Lambda_{\widehat{\bfr},c\bfalpha_i}=p(q^{-1}(Z_{b'}\cap\Lambda_{\widehat{\bfr}+l\bfalpha_i,(c+l)\bfalpha_i})) \text{ if } l<0 \text{).}
\]
As previously, for a general stable point $(x,y,s,t)\in Z_{b'}$, general points in $q(p^{-1}(x,y,s,t))$ (resp.\ $p(q^{-1}(x,y,s,t))$) correspond to stable subrepresentations $M'\subset M(x,y,s,t)$ if $l\geq0$ (resp.\ stable representations $M'\supset M(x,y,s,t)$  if $l<0$) constructed above. Thus, a general point in $Z_b$ is stable, hence $b\in B_{\bfr,c\bfalpha_i}^\st$. It follows that $b'\in f_i^l(B_{\bfr,c\bfalpha_i}^\st)$ (resp.\ $b'\in e_i^{-l}(B_{\bfr,c\bfalpha_i}^\st)$). This finishes the proof.
\end{proof}

We can now bring all the pieces together and complete the construction of $B(C;\lambda)$.

\begin{theorem}[\cref{MainThm/crystalConstruction}]\label{Thm/geometricHWcrystal}
Suppose that $(Q,\bfm)$ is a quiver with multiplicities with associated symmetrisable generalised Cartan matrix $C$ and $\bff$ is a framing vector corresponding to an integral dominant weight $\lambda$ of the associated Kac--Moody algebra $\mathfrak{g}$. Then there is an isomorphism of $\mathfrak{g}$-crystals $B_\bff^\st\simeq B(C;\lambda)$.
\end{theorem}

\begin{proof}
In order to apply \cref{Prop/geometricCriteriaB(lambda)} we need to check the Assumptions \ref{Item1/Prop/geometricCriteriaB(lambda)} - \ref{Item4/Prop/geometricCriteriaB(lambda)}. Assumption \ref{Item1/Prop/geometricCriteriaB(lambda)} holds, as the irreducible component $Z_{b_{\bff,0}}$ consists of a single stable point. Assumptions \ref{Item2/Prop/geometricCriteriaB(lambda)} - \ref{Item4/Prop/geometricCriteriaB(lambda)} follow from \cref{Prop/cokernelsStableRep}, \cref{Lem/framingVanishing} and \cref{Prop/crystalOperatorRange} respectively. Hence by \cref{Prop/geometricCriteriaB(lambda)}, we obtain the desired isomorphism.
\end{proof}

\printbibliography[heading=subbibliography]

\end{document}